\documentclass[10pt]{amsart}

\usepackage{amsmath, amsthm, amssymb, graphicx, enumerate, everypage, datetime, getfiledate, wasysym} 
\usepackage{wrapfig}
\newcounter{citedtheorems}

\newcounter{theoremcounter}

\newtheorem{defn}[theoremcounter]{Definition}
\newtheorem{theorem}[theoremcounter]{Theorem}
\newtheorem{extn}[theoremcounter]{Extension}
\newtheorem*{theorem-m}{Theorem \ref{main-theorem}}

\newtheorem*{theorem-x}{Theorem}
\newtheorem*{theorem-abs1}{Theorem \ref{ind-theorem}}
\newtheorem*{theorem-abs2}{Theorem \ref{a23}}
\newtheorem*{theorem-abs3}{Theorem \ref{ind-new}}
\newtheorem*{theorem-abs4}{Theorem \ref{m1}}
\newtheorem{main-claim}[theoremcounter]{Main Claim}
\newtheorem{thm-lit}[citedtheorems]{Theorem}
\newtheorem{defn-lit}[citedtheorems]{Definition}
\newtheorem{fact-lit}[citedtheorems]{Fact}
\newtheorem{fact}[theoremcounter]{Fact}

\newtheorem{cor}[theoremcounter]{Corollary}
\newtheorem{defn-claim}[theoremcounter]{Definition/Claim}

\newtheorem{concl}[theoremcounter]{Conclusion}
\newtheorem{conv}[theoremcounter]{Convention}
\newtheorem{claim}[theoremcounter]{Claim}

\newtheorem{lemma}[theoremcounter]{Lemma}
\newtheorem{obs}[theoremcounter]{Observation}
\newtheorem{rmk}[theoremcounter]{Remark}

\newtheorem{ntn}[theoremcounter]{Notation}
\newtheorem{disc}[theoremcounter]{Discussion}

\newtheorem{expl}[theoremcounter]{Example}

\newtheorem{qst}[theoremcounter]{Question}

\newtheorem{hyp}[theoremcounter]{Hypothesis}

\newcommand{\br}{\vspace{2mm}}

\newcommand{\vsbr}{\vspace{1mm}}

\newcommand{\ml}{\mathcal{L}}
\newcommand{\tlf}{\trianglelefteq}

\newcommand{\rn}{\operatorname{Range}}

\newcommand{\dom}{\operatorname{Dom}}

\newcommand{\lgn}{\operatorname{lg}}

\newcommand{\cl}{\operatorname{cl}}

\newcommand{\Dqf}{\mathbf{D}}

\newcommand{\tpqf}{\operatorname{tp}_{\operatorname{qf}}} 
\newcommand{\tsqf}{\operatorname{\mathbf{D}}_{\operatorname{qf}}}

\newcommand{\upk}{\Upsilon_{\mk}}
\newcommand{\qftp}{\operatorname{tp}_{\operatorname{qf}}}

\newcommand{\tp}{\operatorname{tp}}

\newcommand{\ocirc}{\astrosun}
\newcommand{\posl}{\operatorname{Pos}}
\newcommand{\negl}{\operatorname{Neg}}

\newcommand{\up}{\Upsilon}

\newcommand{\ii}{\mathbf{i}}

\newcommand{\ts}{\mathbf{S}}

\newcommand{\rstr}{\upharpoonright}

\newcommand{\vp}{\varphi}

\newcommand{\ma}{\mathbf{a}}
\newcommand{\mb}{\mathbf{b}}
\newcommand{\mc}{\mathbf{c}}

\newcommand{\trg}{T_{\mathbf{rg}}}

\newcommand{\GEM}{\operatorname{GEM}}

\newcommand{\xc}{\mathbf{c}}

\newcommand{{\xw}}{\mathbf{w}}
\newcommand{\xr}{\mathfrak{r}}

\newcommand{\xt}{\mathfrak{t}}

\newcommand{\inc}{\operatorname{inc}}
\newcommand{\mk}{\mathcal{K}}
\newcommand{\mcs}{\mathcal{S}}

\title[Generalized EM models...]{A separation theorem for simple theories}

\author[Generalized EM models...]{M. Malliaris and S. Shelah} 

\thanks{\emph{Thanks:} 
Malliaris was partially supported by NSF CAREER award 1553653 and a Minerva research foundation membership at IAS. 
Shelah was partially supported by European Research Council grant 338821 and ISF grant 1838/19.  Both authors thank NSF grant 
1362974 to Shelah at Rutgers, ERC 338821, and NSF-BSF 2051825. This is paper 1149 in Shelah's list.}

\address{Department of Mathematics, University of Chicago, 5734 S. University Avenue, Chicago, IL 60637, USA}
\email{mem@math.uchicago.edu}

\address{Einstein Institute of Mathematics, Edmond J. Safra Campus, Givat Ram, The Hebrew
University of Jerusalem, Jerusalem, 91904, Israel, and Department of Mathematics,
Hill Center - Busch Campus, Rutgers, The State University of New Jersey, 110
Frelinghuysen Road, Piscataway, NJ 08854-8019 USA}
\email{shelah@math.huji.ac.il}
\urladdr{http://shelah.logic.at}

\begin{document}

\begin{abstract}  
This paper builds model-theoretic tools to detect changes in complexity among the simple theories. 
We develop a generalization of dividing, called shearing, which depends on a so-called context $\xc$. 
This leads to defining $\xc$-superstability, a syntactical notion, which includes supersimplicity as a special case.  
The main result is 
a separation theorem showing that for any countable context 
$\xc$  and any two theories $T_1$, $T_2$ such that 
$T_1$ is $\xc$-superstable and $T_2$ is $\xc$-unsuperstable, and for arbitrarily large $\mu$, it is possible to build 
models of any theory interpreting both $T_1$ and $T_2$ whose restriction to $\tau(T_1)$ is $\mu$-saturated and whose 
restriction to $\tau(T_2)$ is not $\aleph_1$-saturated.  (This suggests ``$\xc$-superstable'' is really a dividing line.) 
The proof uses generalized Ehrenfeucht-Mostowski models, and along the way, we clarify the use of these 
techniques to realize certain types while omitting others. 
In some sense, shearing allows us to study the interaction of 
complexity coming from the usual notion of dividing in simple theories and the more combinatorial complexity  
detected by the general definition. 
This work is inspired by our recent progress on Keisler's order, but does not use ultrafilters, rather 
aiming to build up the internal model theory of these classes. 
\end{abstract}

\maketitle

\section{Introduction and motivation}

This paper aims to develop internal model-theoretic tools to detect significant changes in complexity among the simple theories. 

Motivating examples of 
simple theories \cite{Sh:93} include the random graph and random $k$-uniform hypergraphs for 
arbitrary finite $k$. It was subsequently shown that pseudofinite fields, certain higher-order analogues of the triangle-free random graph, and the theory 
ACFA are also simple, see \cite{hrushovski1}, \cite{hrushovski1}, \cite{ch}. The 90s saw a great deal of work on simple theories, as 
recorded in the 2002 survey \cite{GIL}. Still, basic questions about simple theories, such as \ref{qst1} below, remain open. 
The tools we have to detect structural changes in stable theories, such as dividing, still work well in simple theories but the extent to which they explain the whole picture is less clear. 

In the course of our recent work on Keisler's order, a large-scale classification program in model theory which compares theories roughly according to the 
likelihood of saturation in their regular ultrapowers, we made a surprising discovery. 
Although the union of the first two classes in Keisler's order is precisely the stable theories \cite{Sh:a}, 
it turns out that this order has infinitely many classes, already within the simple unstable theories with no nontrivial
dividing, those `near' the random graph \cite{MiSh:1050}. 
A key role was played by what were essentially disjoint unions of the higher analogues of the triangle-free random graph, studied by Hrushovski  
\cite{hrushovski1}. 

The thesis that differences seen by ultrafilters should be significant (as ultrafilters 
are, in some sense, very forgiving) suggests that if a stratification of levels of randomness is appearing in this presumably simple part of the 
map, one should look for an internal explanation. 

It is useful to remember what Keisler's order tells us about the stable theories. 
When the second author proved that the union of the first two classes in Keisler's order is precisely the stable theories, his proof used a 
characterization of the saturated models of stable theories: a model of a stable theory is $\lambda^+$-saturated iff it is $\kappa(T)$-saturated and 
every maximal indiscernible set has size at least $\lambda^+$ \cite[III.3]{Sh:a}. This required developing forking (dividing) and uniqueness of nonforking extensions in stable theories.

The analogous characterization of saturated models of simple theories  
seems to be a real challenge to our understanding:

\begin{qst} \label{qst1}
Give a characterization of the saturated models of simple theories analogous to the theorem that a model of a stable theory is $\lambda^+$-saturated iff it is $\kappa(T)$-saturated and every maximal indiscernible set has size at least $\lambda^+$. 
\end{qst}

Although Question \ref{qst1} remains for the time being open,  in what follows, we will be guided by and will further develop
this core idea of the relation between understanding 
dividing and understanding saturation.

It is also useful to recall some particulars of the higher analogues of triangle-free graphs from \cite{hrushovski1}. 
Let $T_{n,k}$ denote the 
$(n+1)$-free $(k+1)$-hypergraph,  i.e. the
model completion of the theory of a uniform $(k+1)$-ary hypergraph in which there are no $(n+1)$ vertices of which every $(k+1)$ form a hyperedge. 
The triangle-free random graph is not simple, however Hrushovski showed that for $n>k \geq 2$, $T_{n,k}$ is simple with only trivial dividing, 
see \ref{fact15} and \ref{t:trd} below.  So where does the complexity of the $T_{n,k}$s come from? `Amalgamation' is a natural answer, 
and was key to \cite{hrushovski1} and to the property in \cite[1.5]{MiSh:1050}.
Moreover, these amalgamation problems appeared orthogonal to forking.  

However, the methods of the present paper open up a 
different answer. 

We introduce a natural extension of dividing, which we call 
\emph{shearing}, and which includes dividing as a special case.   This definition is developed by looking at dividing in a certain  
canonical context, that of Ehrenfeucht-Mostowski models, and studying realization of types there. 
  In the first part of the paper,  extending an idea from \cite{MiSh:1124}, 
we develop the relation of weak definability of types in generalized Ehrenfeucht-Mostowski models to realizing those types in larger templates. 
There are many parallels to stable phenomena, and various
 definitions which specialize to the familiar ones in the stable case, but they have their own flavor.  
 
In the second part of the paper, we isolate the main mechanism of this correspondence as the definition of \emph{shearing}, which 
a priori makes no reference to $\GEM$ models or to realizing types. 
Dividing involves inconsistency of a formula instantiated along an indiscernible 
sequence; shearing involves inconsistency of a formula instantiated along a generalized indiscernible sequence. 
The definition of shearing involves
choosing an element $I$ from a class $\mk$ of index models, extending the class of linear orders and satisfying certain basic requirements. A \emph{countable context} 
$\mc = (I, \mk)$ is essentially 
a choice of some nontrivial countable $I$ in some allowed $\mk$.  We introduce a notion of a theory being \emph{$\xc$-superstable}, 
essentially the analogue of superstability (or supersimplicity) for the corresponding shearing. 

Theorem \ref{firstmainth} below, the ``separation theorem,'' then explains the connection between shearing and saturation: 
it says essentially that given two theories $T_1, T_2$ and a countable context $\xc$ such that 
$T_1$ is $\xc$-superstable and $T_2$ is $\xc$-unsuperstable, it is possible to build a model (of any theory interpreting both $T_1$ and $T_2$, without loss of generality in disjoint signatures) whose reduct to $\tau(T_1)$ is arbitrarily saturated while the reduct to $\tau(T_2)$ is not even $\aleph_1$-saturated.  
In some sense, we may add weak definitions for all relevant types from $T_1$ while types from $T_2$ remain in this sense undefinable. 
(Alternately, either half of the theorem can be taken as a recipe for building very saturated or very unsaturated models of a given theory according 
to its $\xc$-superstability for a given context.)

Some consequences for $\tlf^*$ are given in \S \ref{s:tlf}. 
In \S \ref{s:simple}, we prove, in some sense, that the focus of shearing is within simplicity. 
\S \ref{other-kappa} outlines natural extensions and some open problems. A companion paper \cite{MiSh:F2061} in progress  
gives a full analysis of the case of the random graph, 
 characterizing the contexts for which it is $\xc$-superstable, and proving that the theories $T_{n,k}$ are strictly more complex in the sense 
of shearing. 

We thank the anonymous referee for many excellent comments and helpful questions. 
We also thank A. Peretz, N. Ramsey, F. Parente, and D. Ulrich.   

\setcounter{tocdepth}{1}
\tableofcontents

\section{Basic notation and definitions}
\setcounter{theoremcounter}{0}

\begin{conv}
All theories are complete and first order unless otherwise stated. 
\end{conv}

\begin{conv}
Given a universal class of models $\mk$, 
we will write ``$J \in \mk$ is $\aleph_0$-saturated'' to mean ``$J$ is countably homogeneous and countably universal for elements of $\mk$,'' 
which makes sense even if $\mk$ is not elementary. 
\end{conv}

We work in the setup of generalized Ehrenfeucht-Mostowski (GEM) models. These methods begin with the EM models of Ehrenfeucht and Mostowski 1965 \cite{EM} and were further developed in e.g.  
Shelah 1978, chapters VII-VIII \cite{Sh:a} and Shelah \cite{Sh:E59}. 
A self-contained introduction may be found in 
our recent paper \cite{MiSh:1124}, \S 3, which takes up the development of these techniques and 
adds the ``G'' for ``generalized" to stress that we may vary the index model $I$, 
see below.  Here we review some basic definitions motivated there, and clarify our assumptions for the 
present paper.

For Ehrenfeucht and Mostowski, index models were linear orders; we will use expansions of linear orders, 
which need not come from an elementary class. (An example from \cite{MiSh:1124} is the class $\mk_\mu$ of linear orders expanded by 
$\mu$ unary predicates which partition the domain; note the ``partition'' requirement implies the class is not elementary.)  
The following general definition will suffice for this paper.\footnote{Item \ref{d:imc}(5) is more than is needed but simplifies our proofs here; asking that 
$\Dqf(\mk)$, the set of quantifier-free types, has amalgamation would suffice.}

\begin{defn}[Index model class] \label{d:imc} \emph{ } Call $\mk$ an index model class, abbreviated \emph{imc}, 
when for some signature $\tau = \tau_\mk \supseteq \{ < \}$,  
\begin{enumerate}
\item  $\mk$ is a class of $\tau$-models, closed under isomorphism, but not necessarily an elementary class. 
\item For each $I \in \mk$, $<^I$ linearly orders $I$.
\item $\mk$ is universal,\footnote{This implies $\mk$ is an $EC(\emptyset, \Gamma)$-class, that is, the set of models of a first order theory which omit some 
$($possibly empty$)$ set of types.  Inversely, 
if $T$ is universal in $\ml(T)$, $\Gamma$ a set of q.f. types then $EC(T, \Gamma)$ is a universal class. As mentioned, $\mk$ need not be an 
elementary class.} so  $I \in \mk$ iff every finitely generated submodel of $I$ is in $\mk$.
\item We could allow partial functions, so for every function symbol $F \in \tau$, there is a predicate $P_F$ which is always interpreted as its domain.  
\item For every $I \in \mk$ there is an $\aleph_0$-saturated $J \in \mk$ with $I \subseteq J$. 

\item $\mk$ is Ramsey, see \ref{d:r1} below.  
\end{enumerate}
\end{defn}

\begin{defn}[$\GEM$ models and proper templates, \cite{Sh:E59} Definition 1.8] \label{d:t-e59} \emph{ }
We say 
$N = \GEM(I, \Phi) = \GEM(I, \Phi, \ma)$ is a generalized Ehrenfeucht-Mostowski model with skeleton ${\ma}$ when
for some vocabulary $\tau = \tau_\Phi$ we have the following.  
\begin{enumerate}
\item $I$ is a model, called the index model. 
\item $N$ is a $\tau_\Phi$-structure and ${\ma} = \{ \bar{a}_t : t \in I \}$ generates $N$. 
\item $\langle \bar{a}_t : t \in I \rangle$ is quantifier free indiscernible in $N$. 
\item $\Phi$ is a template, taking $($for each $n<\omega$$)$ the quantifier free type of 
$\bar{t} = \langle t_0, \dots, t_{n-1} \rangle$ in $I$ to the quantifier free type of 
$\bar{a}_{\bar{t}}$ in $N$. $($So $\Phi$ determines $\tau_\Phi$ uniquely, and also 
a theory $T_\Phi$, the maximal $\tau_\Phi$-theory which holds in every such $N$.$)$ 
\end{enumerate}
$($Note that $\GEM(I, \Phi)$ is not uniquely determined as we have to 
choose the elements of e.g. $N \setminus \rn(\ma)$, but as usually no confusion arises we may omit the additional information. 
So really, ``$N = \GEM(I, \Phi)$'' is a relation.$)$
\end{defn}

When it is useful to specify the skeleton $\ma$ generating a given $\GEM$ model we may display it. Templates 
are simply possible instructions, which may not be `coherent' or give rise to a model; properness says they do. 

\begin{defn} \label{d:proper}
The template $\Phi$ is called proper for $I$ if 
there is $M$ such that $M = \GEM(I, \Phi)$. We say $\Phi$ is proper for a class $\mk$ if $\Phi$ is proper for 
all $I \in \mk$. 
\end{defn}

\begin{defn} \label{conv-upsilon}
Given a class $\mk$, write $\upk$ for the class of templates proper for $\mk$, and write 
$\up$ when $\mk$ is clear from context. 
\end{defn}

\begin{conv} \label{c:nice} 
All templates we consider are assumed to satisfy: 
\begin{enumerate}
\item[(a)] 
nontriviality, i.e. we may add in the $\GEM$ definition the condition
that $\lgn(\bar{a}_t) \geq 1$ and $\langle \bar{a}_t : t \in I \rangle$ is without repetition, 
\item[(b)] $T_\Phi$ is well defined and 
has Skolem functions, where well defined means: 
\begin{enumerate}
\item[(1)] $T_\Phi$ is complete. 
\item[(2)] for every $I \in \mk$, $\langle \bar{a}_{t} : t \in I \rangle$ is indiscernible, not just quantifier-free 
indiscernible, in $\GEM(I, \Phi)$. $($this really follows$)$
\item[(3)] $\GEM(I, \Phi)$ is unique in the sense that it depends, up to isomorphism, 
on $\Phi$ and the isomorphism type of $I$.  More fully: if $N = \GEM(I, \Phi)$ then for some $\ma$, 
$N = \GEM(I, \Phi, \ma)$ and if $N^\prime = \GEM(I, \Phi, \ma^\prime)$ then there is a unique isomorphism from $N$ onto 
$N^\prime$ mapping $\ma$ to $\ma^\prime$, i.e. $\bar{a}_s$ to $\bar{a}^\prime_s$ for $s \in I$. 
\item[(4)] for every $J \supseteq I$ from $\mk$ we have $\GEM(I, \Phi) \preceq \GEM(J, \Phi)$.
\\ More fully, considering the parenthesis in $\ref{d:t-e59}$ we should say: for every $J \supseteq I$ and 
$N_2 = \GEM(J, \Phi)$ there is $N_1 = \GEM(I, \Phi)$ such that $N_1 \subseteq N_2$ hence $N_1 \preceq N_2$. 
Moreover if $N_2 = \GEM(J, \Phi, \ma)$ then $N_1 = \GEM(I, \Phi, {\ma \rstr I})$. $($As in $(3)$, if $N^\prime = \GEM(I, \Phi, \ma^\prime)$ 
then there is a unique isomorphism from $N^\prime$ onto $N_1$ mapping $\ma^\prime$ onto $\ma \rstr I$$)$. 
\end{enumerate}
\end{enumerate}
\end{conv}

\begin{defn} \label{d:order}
Given a class of templates $\Upsilon$, let $\leq_{\Upsilon}$ be the natural partial order 
on $\Upsilon$, that is, $\Phi \leq_{\Upsilon} \Psi$ means that $\tau(\Phi) \subseteq \tau(\Psi)$ and 
$\GEM(I,\Phi) \subseteq \GEM(I, \Psi)$ and $\GEM_{\tau(\Phi)}(I,\Phi) \preceq \GEM_{\tau(\Phi)}(I, \Psi)$.   
We may use $\leq$ when $\Upsilon$ is clear from context. 
\end{defn}

\begin{defn} \label{d:r1}
We say the class $\mk$ is Ramsey when:   given any 
\begin{enumerate}[a)]
\item $J \in \mk$ which is $\aleph_0$-saturated,  
\item model $M$, and 
\item sequence ${\mb} = \langle \bar{b}_t : t \in J \rangle$ of finite sequences from $M$ with the length of $\bar{b}_t$ 
determined by $\tpqf(t, \emptyset, J)$, 
\end{enumerate}
there exists a template $\Psi$ which is proper for $\mk$  such that: 
\begin{enumerate}[i)]
\item $\tau(M) \subseteq \tau(\Psi)$ 
\item $\Psi$ reflects ${\mb}$ in the following sense: 

\noindent for any $s_0, \dots, s_{n-1}$ from $J$, 
\\ any $\vp = \vp(x_0, \dots, x_{m-1}) \in \ml(\tau(M))$, 
\\ and any $\tau(M)$-terms $\sigma_\ell(\bar{y}_0, \dots, \bar{y}_{n-1})$ for $\ell = 0, \dots, m-1$, 

\begin{quotation}
\vsbr
\noindent \underline{if} 
$ M \models \vp[\sigma_0(\bar{b}_{t_0}, \dots, \bar{b}_{t_{n-1}}), \dots, \sigma_{m-1}(\bar{b}_{t_0}, \dots, \bar{b}_{t_{n-1}})] $
\\ for every $t_0, \dots, t_{n-1}$ realizing $\tpqf(s_0~^\smallfrown \cdots ^\smallfrown s_{n-1}, \emptyset, J)$ in $J$,

\vsbr 
\noindent \underline{then} $\GEM(J, \Psi) \models \vp[\sigma_0(\bar{a}_{s_0}, \dots, \bar{a}_{s_{n-1}}), \dots, \sigma_{m-1}(\bar{a}_{s_0}, \dots, \bar{a}_{s_{n-1}})]$
\end{quotation}
where $\langle \bar{a}_s: s \in J \rangle$ denotes the skeleton of $\GEM(J, \Psi)$. 
\end{enumerate}
\end{defn}

\noindent 
We will generally use this definition in the form of Corollary \ref{d:ramsey-exp}.  

\begin{cor} \label{d:ramsey-exp} If $\mk$ is Ramsey,  whenever we are given:
\begin{enumerate}[a)]
\item $J \in \mk$ is $\aleph_0$-saturated
\item $\Phi$ a template proper for $\mk$
\item $M = \GEM(J, \Phi)$ with skeleton $\ma$ 
\item $N^+$, an elementary extension or expansion of $M$, or both   
\end{enumerate}
then there is a template $\Psi$ proper for $\mk$ with $\tau(\Psi) \supseteq \tau(N^+)$ and $\Psi \geq \Phi$.  
Moreover, $\Psi$ reflects $\ma$ in the sense described in $\ref{d:r1}$ ii), with $\ma$ here replacing $\mb$ there. 
\end{cor}

\begin{rmk}
The term ``Ramsey property'' for an index model class is also justified 
by Scow's result that $($in our language$)$ this  
corresponds naturally to the set of finite substructures of elements of the class 
being a Ramsey class in the sense of  Ne\v{s}et\v{r}il \cite{ns} and
of Kechris-Pestov-Todor\v{c}evi\'{c} \cite{kpt}. See Scow \cite{scow2} Theorem 4.31. 
\end{rmk}

The last definition of this section will be crucial for the rest of the paper.  Recall the definition of ``index model class,'' \ref{d:imc}, which 
had various mild restrictions on which classes of index models we may consider. 
For many of our arguments we will fix not only some index model class $\mk$ but some particular $I \in \mk$, and the following conditions 
ensure in various ways that our $I$ is not trivial.

\begin{defn}[Context] \label{d:context}
A context $\xc$ is a tuple $(I, \mk) = (I_\xc, \mk_\xc)$ such that $\mk$ is an index model class and $I \in \mk$, and in addition: 

\begin{enumerate}
\item  If $\tau(\mk)$ includes function symbols, then in addition we require that $I = \cl(I)$. 

\item $I_\xc$ is \emph{nontrivial}, meaning that $I \neq \cl(\bar{t}, I)$ for every finite $\bar{t} \subseteq I$. 

\item $I$ is \emph{reasonable}, meaning that whenever 
$I \subseteq J$ where $J \in \mk$ is $\aleph_0$-saturated, if $\bar{t} \in {^{\omega>}I}$, $s \in J$ and 
\[ \mbox{ \emph{(for all $r \in J$) ($\tpqf(r,\bar{t}, J) = \tpqf(s, \bar{t}, J)$ implies $r=s$) }} \]
 then $s \in \cl_I(\bar{t})$.

\item $I$ is \emph{non-1-trivial}, meaning that whenever $I \subseteq J$ where $J \in \mk$ is $\aleph_0$-saturated, 
$\bar{t} \in {^{\omega>}I}$, $s \in J$ and $s \notin \cl(\bar{t})$ 
then 
\[ \{  r \in J ~ : ~ \tpqf(r,\bar{t}, J) = \tpqf(s,\bar{t}, J) \}  \mbox{ is infinite.} \] 
\end{enumerate}
\end{defn}

\begin{ntn} \label{notation15}
Given a context $\xc$,  which fixes $\mk = \mk_\xc$ and $\Upsilon = \Upsilon_\mk$, and given a theory $T$, 
\begin{enumerate}
\item[(a)] Let $\Upsilon[T]$ be the class of $\Phi \in \Upsilon$ such that $\tau_\Phi \supseteq\tau(T)$ and 
$T_\Phi \supseteq T$ and $T_\Phi$ has Skolem functions for $T$. 
\item[(b)] Let $\Upsilon[\lambda, T]$ be the class of $\Phi \in \Upsilon[T]$ such that $\tau(\Phi)$ has size $\leq \lambda$.
\end{enumerate}
\end{ntn}

\begin{ntn}  \label{notation14}
Given any linearly ordered set $I$, 
let $\inc_n(I)$ denote the set of strictly increasing $n$-element sequences from $I$, and let $\inc(I) = \bigcup_n \inc_n(I)$. 
\end{ntn}

\vspace{5mm}

\section{$\mk$-indiscernible sequences}
\setcounter{theoremcounter}{0}

This section discusses $\mk$-indiscernible sequences, for a given index model class $\mk$, Definition \ref{d:imc} above.   
These were introduced in \cite{Sh:a} and have an interesting and varied history in the model theoretic 
literature, both in works of the second author and many others.  
Notably, 
the idea that generalized indiscernibles could give insight into model-theoretic dividing lines has been developed in a different direction by 
Scow \cite{scow2} and Guingona-Hill-Scow \cite{ghs}.

Readers familiar with 
some such definition are nonetheless encouraged to read the remark after Definition \ref{d:current}. 

To start, for the purposes of discussion, the familiar definition of an indiscernible sequence may be written as follows. 

\begin{defn} \label{d:ind-a}
Suppose we are given an ordered set $(I, <)$, a model $N$ of $T$, $A \subseteq N$, and a map $f : I \rightarrow {^{\omega>}N}$. For each 
$t \in I = \dom (f)$, write $\bar{b}_t$ for $f(t)$, so the image of $f$ is the sequence $\mb = \langle \bar{b}_t : t \in I \rangle$. We say $\mb$
is an \emph{indiscernible sequence over $A$} when it satisfies: 
for all $k<\omega$, all $t_0,\dots, t_{k-1}$ and $t^\prime_0, \dots, t^\prime_{k-1}$ from $I$,  if 
\begin{equation}
\label{ind-1} 
\tpqf({{t}_0}^\smallfrown {{t}_1}^\smallfrown \cdots ^\smallfrown {t}_{k-1}, \emptyset, I) = 
\tpqf({t^\prime_0}^\smallfrown {t^\prime_1}^\smallfrown \cdots ^\smallfrown t^\prime_{k-1}, \emptyset, I) 
\end{equation}
then in $N$, or equivalently in the monster model $\mathfrak{C} = \mathfrak{C}_T$, 
\begin{equation} 
\tp_{\tau(T)}({\bar{b}_{{t}_0}} ~^\smallfrown ~ {\bar{b}_{{t}_1}} ~ ^\smallfrown ~ \cdots ~^\smallfrown ~{\bar{b}_{{t}_{k-1}}}, A, \mathfrak{C}) 
= \tp_{\tau(T)}({{\bar{b}_{t^\prime_0}}} ~ ^\smallfrown~ {\bar{b}_{t^\prime_1}} ~ ^\smallfrown  ~\cdots ~ ^\smallfrown{\bar{b}_{t^\prime_{k-1}}}, A, \mathfrak{C}). 
\end{equation}
\end{defn}

\vspace{5mm}

\noindent In the following key definition,  we choose an $I$ which may be an expansion of a linear order,  
the domain of $f$ changes from $I$ to ${^{\omega>}I}$,  
and \ref{d:ind-a}(\ref{ind-1}) is updated in the natural way (note the inset line beginning ``$\ell < k$'' in \ref{d:current} is trivially satisfied when the 
$t$'s are singletons). 

\br

\begin{defn}[$\mk$-indiscernible sequence] \label{d:current} \emph{ } 
Suppose we are given a context $\xc$, thus $I = I_\xc$ and $\mk = \mk_\xc$. 
Suppose we are given a model $N$ of $T$, $A \subseteq N$, and a map $f: {^{\omega >}I} \rightarrow {^{\omega >}N}$. 
For each $\bar{t} \in \dom(f)$, write $\bar{b}_{\bar{t}}$ for $f(\bar{t})$, so the image of $f$ is the sequence 
$\mb = \langle \bar{b}_{\bar{t}} : \bar{t} \in {^{\omega >}I} \rangle$. We say $\mb$ is a \emph{$\mk$-indiscernible sequence over $A$} 
when it satisfies:
for all $k < \omega$, all $\bar{t}_0, \dots, \bar{t}_{k-1}$ and all $\bar{t}^\prime_0,\dots, \bar{t}^\prime_{k-1}$ from ${^{\omega>}I}$, if 
\[ \ell < k \implies \lgn(\bar{t}_\ell) = \lgn(\bar{t}^\prime_\ell)  
 ~ \mbox{ and } \]
\[ \tpqf(  {\bar{t}_0}  ~ ^\smallfrown {\bar{t}_1} ~^\smallfrown~ \cdots ~^\smallfrown ~\bar{t}_{k-1}, \emptyset, I) = 
\tpqf(  {\bar{t}^\prime_0} ~^\smallfrown ~{\bar{t}^\prime_1} ~^\smallfrown ~\cdots ~ ^\smallfrown  ~ \bar{t}^\prime_{k-1}, \emptyset, I) \] 
then in $N$, or equivalently in the monster model $\mathfrak{C} = \mathfrak{C}_T$, 
\[ \tp_{\tau(T)}({\bar{b}_{{\bar{t}}_0}} ~^\smallfrown ~ {\bar{b}_{{\bar{t}}_1}} ~ ^\smallfrown ~ \cdots ~^\smallfrown ~{\bar{b}_{{\bar{t}}_{k-1}}}, A, \mathfrak{C}) = 
\tp_{\tau(T)}({{\bar{b}_{\bar{t}^\prime_0}}} ~ ^\smallfrown~ {\bar{b}_{\bar{t}^\prime_1}} ~ ^\smallfrown  ~\cdots ~ ^\smallfrown{\bar{b}_{\bar{t}^\prime_{k-1}}}, A, \mathfrak{C}). \]
\end{defn}

\br
\noindent Definition \ref{d:current} improves the range of \ref{d:ind-a} substantially. A very useful and less obvious way it does so  
may be observed as follows. If $\mk$ is a class of linear orders and $I \in \mk$, and if we are given a function $f$ and a sequence $\mb$ satisfying 
Definition \ref{d:ind-a}, we may extend the domain of $f$ naturally to ${^{\omega>}I}$ by setting 
\begin{equation}
\label{e:same} \bar{b}_{\bar{t}} =  \bar{b}_{t_0} ~^\smallfrown ~\cdots ~ ^\smallfrown \bar{b}_{t_{\ell-1}}         \mbox{ when } \bar{t}  = \langle t_0, \dots, t_{\ell-1} \rangle 
\end{equation}
to generate a sequence $\mb$ satisfying \ref{d:current}. 
However, Definition \ref{d:current} doesn't ask that something like (\ref{e:same}) be true. 
A priori, in \ref{d:current}, 
\begin{equation} \label{e:caveat} 
\bar{b}_{t^\smallfrown s} \mbox{ may not be equal to } \bar{b}_t~^\smallfrown \bar{b}_s. 
\end{equation}
Consider the following family of examples (\ref{d:earlier}), which will require a few definitions.  

\begin{ntn} Given a context $\xc$, $\Dqf(I)$ is the set of quantifier-free types\footnote{It would be more consistent 
with standard notation, if a little less readable, to write $\mathbf{D}_{\operatorname{qf}}(I)$.}
of strictly increasing 
finite sequences of elements of $I$.  
\end{ntn} 

Let us name the set of tuples in $I$ sharing a quantifier-free type ($\inc_n$: \ref{notation14}). 

\begin{defn}
For $I \in \mk$ and $\xr \in \Dqf(I)$, and implicitly $n = n(\xr)$, let 
\[ Q^I_\xr = Q^I_{\xr, n} = \{ \bar{t} : \bar{t} \in \inc_n(I), \tpqf(\bar{t}, \emptyset, I) = \xr \} \]
be the set of realizations of $\xr$ in $I$. 
\end{defn}

\begin{defn}[$\xr$-indiscernible sequence]
\label{d:earlier} \emph{ } 
Suppose we are given a context $\xc$, thus $\mk = \mk_\xc$ and $I = I_\xc$ .  Suppose we are given a theory $T$, a model 
$N \models T$, $A \subseteq N$, a type $\xr \in \Dqf(I)$, and 
a map $f: Q^I_\xr \rightarrow {^{\omega >}N}$.  
For each $\bar{t} \in Q^I_\xr$ 
write $\bar{b}_{\bar{t}}$ for $f(t)$. 
We say 
\[  \mb =  \langle \bar{b}_{\bar{t}} : \bar{t} \in Q^I_\xr  \rangle \] 
is a \emph{$\xr$-indiscernible sequence over $A$} when: 

\begin{enumerate}
\item[\emph{(a)}] for all $\bar{t} \in Q^I_\xr$, $\lgn(\bar{b}_{\bar{t}})$ is finite and constant. 
\item[\emph{(b)}] for all finite $k$, if $\bar{t}_0, \dots, \bar{t}_{k-1},  {\bar{t}^\prime_0}, \dots, {\bar{t}^\prime}_{k-1}  \in Q^I_\xr$ 
and 
\[ \tpqf({\bar{t}_0}~^\smallfrown ~{\bar{t}_1} ~ ^\smallfrown ~ \cdots ~ ^\smallfrown \bar{t}_{k-1}, \emptyset, I) = 
\tpqf({\bar{t}^\prime_0} ~^\smallfrown ~{\bar{t}^\prime_1} ~ ^\smallfrown ~ \cdots  ~ ^\smallfrown  ~ \bar{t}^\prime_{k-1}, \emptyset, I) \] 
then in $N$, or equivalently in the monster model $\mathfrak{C} = \mathfrak{C}_T$, 
\[ \tp_{\tau(T)}({\bar{b}_{{\bar{t}}_0}} ~^\smallfrown ~ {\bar{b}_{{\bar{t}}_1}} ~ ^\smallfrown ~ \cdots ~^\smallfrown ~{\bar{b}_{{\bar{t}}_{k-1}}}, A, \mathfrak{C}) = 
\tp_{\tau(T)}({{\bar{b}_{\bar{t}^\prime_0}}} ~ ^\smallfrown~ {\bar{b}_{\bar{t}^\prime_1}} ~ ^\smallfrown  ~\cdots ~ ^\smallfrown{\bar{b}_{\bar{t}^\prime_{k-1}}}, A, \mathfrak{C}). \]
\end{enumerate}
\end{defn}

\begin{obs}  
Definition \ref{d:earlier} can naturally be considered as a special case of Definition \ref{d:current}. 
\end{obs}

\begin{proof}
Extend $f$ in \ref{d:earlier} to ${^{\omega >}I}$ by setting $f(u) = \emptyset$ for all $u \in {^{\omega>}I} \setminus Q^I_\xr$. 
\end{proof}

As another example, $\mk$-indiscernible sequences arise naturally in $\GEM$ models. 

\begin{expl}
For any context $\mc = (I, \mk)$ and any $M = \GEM(I, \Phi, \ma)$, the template $\Phi$ determines
an $f$ showing that the skeleton $\ma$ is a 
$\mk$-indiscernible sequence. 
\end{expl}

\begin{rmk}
In the example of a skeleton, of course, equation (\ref{e:same}) above does hold; see also convention \ref{c53}.
\end{rmk}

So far we have been careful to write $\bar{t}$ for finite tuples from $I$ of length possibly $>1$, as distinguished from singletons $t \in I$, in order to clearly make the point in 
equation (\ref{e:caveat}), p. \pageref{e:caveat} above.  However, for the remainder of the paper, it will greatly simplify readability to 
also allow $s,t$ to range over elements of 
$\inc(I)$.  

\begin{conv}[Dropping some overlines]
Beginning in $\S \ref{s:wd}$ and to the end of the paper, unless otherwise stated, we allow $s,t$ to range over elements of 
${^{\omega>}I}$, not just $I$. For example, referring to sequences as in $\ref{d:earlier}$, we may write 
\[  \mb =  \langle \bar{b}_{{t}} : {t} \in Q^I_\xr  \rangle \] 
when $n = n(\xr)$ is not necessarily $1$.   $($This convention doesn't mean we won't continue to use overlines; it just means that the 
lack of an overline doesn't mean the length is $1$.$)$
\end{conv} 

Classically in model theory, a main use of indiscernible sequences is in the definition of dividing, 
and so we may expect that the more robust notion of indiscernible sequence 
would give us a more finely calibrated notion of dividing. This will be developed in \S \ref{s:fd}, after a section which may  
justify some particulars of that definition.

\vspace{5mm}

\section{Weak definability and saturation}  \label{s:wd}
\setcounter{theoremcounter}{0}

Developing an idea from \cite{MiSh:1124} \S 9,\footnote{the reader does not need to have seen that paper to follow the present section.} 
this section shows that $\GEM$-models reveal a useful weakening of the phenomenon of definability of types from stable theories. 
Moreover, we will see that existence of these weak definitions may be characterized in terms of realization and omission of types 
in extensions of the given $\GEM$ model, and so is tightly connected to the problem of building saturated models in this setup.

\begin{conv} \label{c53}
When $M = \GEM(I, \Phi)$ with skeleton $\ma = \langle \bar{a}_t : t \in I \rangle$, then whenever $\bar{t} = \langle t_0, \dots, t_{k-1} \rangle \in \inc(I)$, 
\[ \bar{a}_{\bar{t}} \mbox{ \emph{ abbreviates } } \bar{a}_{t_0} ~ ^\smallfrown~ \bar{a}_{t_1} ~ ^\smallfrown  ~\cdots ~ ^\smallfrown  \bar{a}_{t_{k-1}}. \]
\end{conv}

\vspace{4mm}

\noindent 

\begin{disc} \label{disc:enum}
\emph{To motivate the first main definition of the section, Definition \ref{d:gd}, suppose we are given a context $\xc = (I, \mk)$, 
a complete theory $T$, and $M = \GEM(I, \Phi) \models T$ with skeleton $\ma = \langle \bar{a}_t : t \in I \rangle$. Suppose 
$p \in \ts_{\tau(T)}(M)$ is a type or a partial type, so we may enumerate it as} 
\begin{equation}
\label{enum-1}
\langle \vp_\alpha(x,\bar{b}_\alpha) : \alpha < \kappa \rangle
\end{equation}
\emph{for some $\kappa$ depending on $p$. Since we are 
in a $\GEM$ model, we may write a more informative version of (\ref{enum-1}),}  
\begin{equation}
\label{eq:enum} \langle \vp_\alpha(x, \bar{\sigma}_\alpha (\bar{a}_{\bar{t}_\alpha})) ~:~ \alpha < \kappa \rangle 
\end{equation}
\emph{where each $\bar{\sigma}_\alpha$ abbreviates some finite sequence of $\tau(\Phi)$-terms 
$\langle \sigma_\ell (\bar{y}_0, \dots, \bar{y}_{n-1}) : \ell <  m(\alpha) \rangle$, and 
$\bar{t}_{\alpha} = \langle t_0, \dots, t_{n-1} \rangle \in \inc_n(I)$, so 
$\bar{a}_{\bar{t}_{\alpha}}$ is a sequence 
from the skeleton; thus,  $\bar{\sigma}_\alpha (\bar{a}_{\bar{t}_\alpha})$ abbreviates 
$\langle \sigma_\ell(\bar{a}_{\bar{t}_\alpha}) : \ell < m(\alpha) \rangle$.   [In order to evaluate this expression, it should of 
course be the case that for each $i < n$, $\lgn(\bar{a}_{t_i}) = \lgn(\bar{y}_i)$.]
 The choice of $\bar{\sigma}_{\alpha}$, $\bar{a}_{\bar{t}_{\alpha}}$ need not be unique; any choice with the property that 
$\langle \sigma_\ell(\bar{a}_{\bar{t}_{\alpha}}) : \ell < m(\alpha) \rangle$ 
evaluates correctly in $M$ to $\bar{b}_\alpha$, will do.} 
\end{disc}

\begin{defn} \label{d:detailed}
Given a context $\xc = (I, \mk)$, $M = \GEM(I, \Phi) = \GEM(I, \Phi, \ma)$ and a type $p \in \ts(M)$, 
call any enumeration of $p$ satisfying $(\ref{eq:enum})$ of Discussion $\ref{disc:enum}$ a \emph{detailed enumeration}. 
\end{defn}

\begin{disc}
\emph{Continuing \ref{disc:enum}, each item in the sequence $(\ref{eq:enum})$ has three natural ingredients: 
the formula $\vp_\alpha$, the sequence of $\tau(\Phi)$-terms $\bar{\sigma}_\alpha$, and 
$\xr_\alpha = \qftp(\bar{t}_\alpha, \emptyset, I)$. The move from $\bar{t}_\alpha$ to its quantifier-free type $\xr_\alpha$ potentially loses 
information.  Our question is whether this is serious, i.e. whether there is a partial function} 
\begin{equation} 
\label{eq:f}
F : (\mbox{ $\tau(T)$-formulas }) \times (\mbox{ finite sequences of $\tau(\Phi)$-terms }) \times \Dqf(I)  \longrightarrow \{ 0, 1 \} 
\end{equation}
\emph{such that given any $J$ with $I \subseteq J \in \mk$, the set of formulas} 
\begin{equation} \label{eq:f2}
\{  \vp(x, \bar{\sigma}(\bar{a}_{\bar{t}}))^{\ii} :  (\vp, \bar{\sigma}, \xr) \in \dom(F),  ~\bar{t} \in \inc(J), \tpqf(\bar{t}, \emptyset, J) = \xr, 
~F( \vp, \bar{\sigma}, \xr) = \ii \}
\end{equation}
\emph{when evaluated in $N = \GEM(J, \Phi)$, is consistent and extends $p$.} 
\end{disc}

We will formally define such functions $F$ in \ref{d:gd} below after a few additional remarks and adjustments. 

First, why do we consider all larger $J$'s?  The deeper answer will be that, just as the usual definability of types is most useful in controlling extensions of the 
given type to larger models,  here we will use $F$ in applications of  
\ref{d:ramsey-exp}, which will require $J$ to be sufficiently saturated. The simpler, initial answer is that for many natural $I$, restricting to $I = J$ gives $F$ trivially, as the 
next example explains.  

\begin{defn} \label{d:sep} 
Let $\xc = (I, \mk)$ be a context.  We say $I$ is \emph{separated} when $s \neq t \in I$ implies 
$\tpqf(s, \emptyset, I) \neq \tpqf(t, \emptyset, I)$.
\end{defn}

An example of \ref{d:sep} which played a key role in {\cite{MiSh:1124} \S 5}: for a given infinite $\mu$, $\mk_\mu$ is the class of linear orders 
expanded by $\mu$ unary predicates which partition the domain, which is known to be an index model class. A separated $I \in \mk_\mu$ is one in which each element of $I$ has its own color. 

\begin{rmk} 
When $I$ is separated, each $\bar{t} \in \inc(I)$ is the unique realization of its quantifier-free type 
$\xr = \tpqf(\bar{t}, \emptyset, I)$, so for the case $I = J$, a function $F$ following $(\ref{eq:f})$ exists trivially, and the more  
interesting question concerns $J \supseteq I$. 
\end{rmk}

One more example will explain the appearance of the finite $\bar{t}_*$ in Definition \ref{d:gd}. 

\begin{expl} \label{ex-orders}
\emph{Let $\mk$ be the class of infinite linear orders, and $I = (\mathbb{Q}, <)$. Let $T$ be the theory of an equivalence relation with infinitely 
many infinite classes. Choose $M = \GEM(I, \Phi, \ma)$ to be a countable model with $\ma = \langle a_t : t \in I \rangle$ a sequence of elements from 
distinct equivalence classes. By our assumption \ref{c:nice}, there are Skolem functions for $T$, say, $\langle f_i : i < \omega \rangle$ in $\tau(\Phi)$ 
interpreted so that $\langle f^M_i(a_t) : i < \omega \rangle$ enumerates the equivalence class of ${a}_t$.   Let $b$ be any element of $M$ 
and let $p$ be the partial type $\{ E(x,b) \}$. Then we may choose a detailed enumeration of $p$, say, 
\[ p = \langle~ E(x,f_i({a}_{t_*}))~ \rangle \]
for some $i = i_p <\omega$ and some $t_* = t_p \in I$. But since any two $t, t^\prime$ in $I$ have the same quantifier-free type, 
no function $F$ satisfying (\ref{eq:f})-(\ref{eq:f2}) above exists.    This is easily solved by allowing $F$ to depend on some finite sequence from $I$, 
here $t_*$.} 
\end{expl}

\begin{defn} 
For $\bar{t}_* \in \inc(I)$, let $\Dqf(I, \bar{t}_*)$ denote the set of quantifier-free types over $\bar{t}_*$ of strictly increasing finite sequences 
of elements of $I$, i.e. 
\[ \Dqf(I, \bar{t}_*) = \{   \tpqf(\bar{t}, \bar{t}_*, I) : \bar{t} \in \inc(I) \}. \]  
\end{defn}

\vspace{4mm}

We arrive at the main definition of the section. 

\begin{defn}[Weakly definable type] \label{d:gd}
Suppose we are given a context $\xc = (I, \mk)$, a complete theory $T$,  $M = \GEM(I, \Phi) = \GEM(I, \Phi, \ma) \models T$ and a partial type or type $p \in \ts_{\tau(T)}(M)$, of $\ml(\tau_T)$. 
Say $p$ is \emph{weakly definable} when there exist 
\begin{enumerate}
\item[(a)] a detailed enumeration 
$\langle \vp_\alpha(x, \bar{\sigma}_\alpha (\bar{a}_{\bar{t}_\alpha})) ~:~ \alpha < \kappa \rangle$ of $p$, where $\kappa = |p|$, 
\item[(b)] a finite sequence $\bar{t}_* \in \inc(I)$,  
\item[(c)] a partial function $F$ depending on $\bar{t}_*$, such that 
\[ F : \{ \bar{t}_* \} \times  (\mbox{ $\tau(T)$-formulas }) \times (\mbox{ finite sequences of $\tau(\Phi)$-terms }) \times~ \Dqf(I, \bar{t}_*)  \longrightarrow \{ 0, 1 \} \]
and for some $\aleph_0$-saturated $J$ with $I \subseteq J \in \mk$, the set of formulas
\[ q = \{  \vp(x, \bar{\sigma}(\bar{a}_{\bar{t}}))^{F( \bar{t}_*, \vp, \bar{\sigma}, \xr)} :  (\bar{t}_*, \vp, \bar{\sigma}, \xr) \in \dom(F),  ~\bar{t} \in \inc(J), \tpqf(\bar{t}, \bar{t}_*, J) = \xr \} \] 
when evaluated in $N = \GEM(J, \Phi)$, is consistent and extends $p$.
\end{enumerate} 
\end{defn}

\begin{ntn}
In the context of $\ref{d:gd}$, we may also write``$p$ is weakly definable over $\bar{t}_*$'' or ``$p$ has a weak definition over $\bar{t}_*$'' to emphasize the choice of the finite $\bar{t}_*$.
\end{ntn}

\begin{obs} \label{o:larger}
If $p$ is a partial type of $\GEM(I, \Phi)$ and is weakly definable, then $p$ remains weakly definable in $\GEM(I, \Psi)$ for any 
$\Psi$ with $\Phi \leq \Psi \in \Upsilon[T]$, as witnessed by the same $\bar{t}_*$ and $F$. 
\end{obs}

\begin{obs} \label{o:corresp}
\emph{In the context of $\ref{d:gd}$, note that it follows from the definition of $\GEM$-model that if $p$ has a weak definition over some $\bar{t}_* \in \inc(I)$, 
and if ${\bar{s}^*}$ is any other sequence from $J$ with $\tpqf({\bar{s}^*}, \emptyset, J) = \tpqf(\bar{t}_*, \emptyset, J)$,  then 
the set of formulas}
\[ \{  \vp(x, \bar{\sigma}(\bar{a}_{\bar{t}}))^{F( {\bar{s}^*}, \vp, \bar{\sigma}, \xr)} :  ({\bar{s}^*}, \vp, \bar{\sigma}, \xr) \in \dom(F),  ~\bar{t} \in \inc(J), \tpqf(\bar{t}, {\bar{s}^*}, J) = \xr \} \]
\emph{when evaluated in $N = \GEM(J, \Phi)$, is consistent.  Moreover, this consistent set of formulas 
extends a natural analogue of $p$, namely, the type obtained by 
replacing every occurrence of $\bar{t}_*$ in the given detailed enumeration of $p$ by ${\bar{s}^*}$.} 
\end{obs}

\begin{rmk}
In Definition $\ref{d:gd}$ the particular choice of $J$ will not matter, only that it is $\aleph_0$-saturated and extends $I$.  
We could have stated the definition for some, equivalently every, $\aleph_0$-saturated $J \supseteq I$ from $\mk$. 
\end{rmk}

\begin{disc}
\emph{ In Definition \ref{d:gd}, existence of a weak definition depends on $I, \mk, \Phi$, 
not only on the type. 
We might also say it is the extension $q$ of $p$ which has the weak definition.  When such a weak definition exists, 
then for each $J$ the extension $q \supseteq p$ we obtain is unique.  (We aren't asserting this is  
independent of the choice of  
the detailed enumeration, and we have also left open the possibility of varying the domain of $F$ to include e.g. formulas not used in $p$ -- 
but once $F$ is given, for each larger $J$ there is no ambiguity.) It may not be a complete type over $N$, since in $J$ there 
may be many $\bar{t}$'s which do not realize 
any type in $\Dqf(I)$ and so are never used, for example, 
if $I$ is separated and $J$ is $\aleph_0$-saturated, $J$ will contain many finite sequences in which distinct 
elements have the same quantifier-free 1-types, and these have no analogue in $I$.  However, if $I$ is $\aleph_0$-saturated, then $q$ will be a complete type. }
\end{disc}

\br

\begin{claim}[Definable implies weakly definable, for formulas] \label{c:stable-def}
Suppose $\xc = (I, \mk)$ is a context, $T$ a complete theory,  $M = \GEM(I, \Phi) = \GEM(I, \Phi, \ma) \models T$. 
\begin{enumerate}
\item[(a)] Suppose $\Delta = \{ \vp, \neg \vp \}$ for some stable formula $\vp$ of $T$. Any type $p \in \ts_{\Delta}(M)$ 
has a weak definition over some finite $\bar{t}_* \in \inc(I)$. 
\item[(b)]  In the previous item, $\Delta$ may be of any finite size as long as it contains only formulas which are stable in $T$. 
\end{enumerate}
\end{claim}

\begin{proof} 
Since definitions operate formula-by-formula and the concatenation of finitely many finite $\bar{t}_*$'s is still finite, it will suffice to prove case (a). So let us assume $\Delta = \{ \vp(x,\bar{y}), \neg \vp(x,\bar{y}) \}$ where $\vp$ is a stable formula, and 
$\ell(x)$ need not be 1. 

As $\vp$ is stable and $M$ is a model, there is a formula $\theta = \theta(\bar{y}, \bar{z})$ and a sequence of elements 
$\bar{c} \in {^{\lgn(\bar{z})}M}$ such that for all $\bar{b} \in {^{\lgn(\bar{y})} M}$, 
\[ \vp(x,\bar{b}) \in p ~ \mbox{ if and only if } ~ M \models \theta(\bar{b}, \bar{c}). \]
Fix some sequence $\bar{\sigma}_*$ of $\tau(\Phi)$-terms and some $\bar{t}_* \in \inc(I)$ so that evaluated in $M$, 
\[ \bar{c} = \bar{\sigma}_*(\bar{a}_{\bar{t}_*}). \]
Fix any detailed enumeration of $p$ : 
\[  \langle \vp(x, \bar{\sigma}_\alpha (\bar{a}_{\bar{t}_\alpha}))^{\ii_\alpha} ~:~ \alpha < \kappa \rangle. \]
Consider the function $F$ given by  
\[ (\bar{t}_*, \vp, \bar{\sigma}_\alpha, \tpqf(\bar{t}_\alpha, \bar{t}_*, I))~ \mapsto ~ \ii_\alpha.  \]
Fix any $\aleph_0$-saturated $J \supseteq I$ from $\mk$ and we would like to show the application of $F$ 
defines a consistent $q \supseteq p$. 
Recall from $\ref{c:nice}(3)$ that $\GEM_{\tau(T)}(I, \Phi) \preceq \GEM_{\tau(T)}(J, \Phi)$.  
So if $\vp(x,\bar{\sigma}_\alpha(\bar{a}_{\bar{t}_\alpha}))^{\ii_\alpha} \in p$, then for any other $\bar{t}^\prime \in \inc(J)$ such that 
\[ \tpqf(\bar{t}^\prime, \bar{t}_*, J) = \tpqf(\bar{t}_{\alpha}, \bar{t}_*, I) \]
we have that in $N = \GEM_{\tau(T)}(J, \Phi)$, again recalling $\ref{c:nice}(2)$, 
\[ \tp_{\tau(T)}(\bar{a}_{\bar{t}^\prime}, \bar{a}_{\bar{t}_*}, N) = \tp_{\tau(T)}(\bar{a}_{\bar{t}_{\alpha}}, \bar{a}_{\bar{t}_*}, N). \]
In particular, 
\[ \psi(\bar{a}_{\bar{t}^\prime}, \bar{a}_{\bar{t}_*}) := \theta(\bar{\sigma}_{\alpha}(\bar{a}_{\bar{t}^\prime}), 
\bar{\sigma}_*(\bar{a}_{\bar{t}_*})) =  \theta(\bar{\sigma}_{\alpha}(\bar{a}_{\bar{t}^\prime}), \bar{c}) \]
will hold in $N$ if and only if 
\[ \psi(\bar{a}_{\bar{t}_{\alpha}}, \bar{a}_{\bar{t}_*}) := \theta(\bar{\sigma}_{\alpha}(\bar{a}_{\bar{t}_{\alpha}}), 
\bar{\sigma}_*(\bar{a}_{\bar{t}_*})) = \theta(\bar{\sigma}_{\alpha}(\bar{a}_{\bar{t}_\alpha}), \bar{c})\]
holds in $N$, so if and only if $\ii_\alpha = 1$. 
So $F$ agrees with the definition given by $\theta(\bar{y}, \bar{c})$, thus its output will be consistent. 
\end{proof}

\begin{disc}
\emph{In the precursor to this paper \cite{MiSh:1124} \S 9 we summarized the main results proved there by suggesting a definition corresponding to 
(in the present notation)
weak definitions over the empty set.  The proof of \cite{MiSh:1124}, Claim 5.10 
there established that for $T = \trg$ the theory of 
the random graph, $\mk = \mk_\mu$ the class of 
linear orders expanded by $\mu$ unary predicates which partition the domain, and a context $\mc = (I, \mk)$ where 
$I$ is separated,  we have that in any $M = \GEM(I, \Phi) \models T$, any partial $\Delta$-type over $M$ for 
$\Delta = \{ R(x,y), \neg R(x,y) \}$ has a weak definition over the empty set.}
\end{disc}

\begin{concl} 
\label{c:stx} Weak definability of $\vp$-types is strictly weaker than definability of $\vp$-types, since a $\vp$-type over a model is definable if and only if 
$\vp$ is stable. 
\end{concl}

\begin{disc}
\emph{The extension of stability in \ref{c:stx} requires looking locally. Notice we have \emph{not} called a type weakly definable when each of its formulas is.  Rather, we require a single finite sequence 
$\bar{t}_*$ which works for the entire type. (What if each formula is weakly definable but the 
type is not? Then there is no problem in realizing each $\vp$-type in some larger $\Psi$, but we won't be able to realize the 
entire type at the same time.) 
This is justified by Claim \ref{c:143} below, and indeed, the careful reader may guess that 
non-superstability, suitably extended, will have an important role to play in what follows.  } 
\end{disc}
\br

We now connect weak definability to the construction of saturated models. 
For the remainder of the section, let the following be arbitrary but fixed. 

\begin{hyp} \label{m2} \emph{ For the rest of the section,  }
\begin{enumerate}
\item $\xc$  a context, so $I = I_\xc$ and $\mk = \mk_\xc$ are given.   
\item $T$  a complete first-order theory. 
\item $\Phi \in \Upsilon = \Upsilon_\mk[T]$, recalling notation $\ref{notation15}$. 
\item $M = \GEM_{\tau(T)}(I, \Phi)$. 
\item ``there exists $\Psi \geq \Phi$'' always means $\Psi \in \Upsilon[T]$. 
\end{enumerate}
\end{hyp}

\begin{claim} \label{c:142}
Let $p = p(\bar{x})$ be a partial type in $M$. Suppose $p$ has a weak definition over some 
finite $\bar{t}_* \subseteq I$. Then there exists $\Psi \in \Upsilon$, $\Psi \geq \Phi$ such that $p$ is realized in $\GEM_{\tau(T)}(I, \Psi)$. 
\end{claim}

\begin{rmk} 
In the special case when $\bar{t}_*$ is empty, this was noted in \cite{MiSh:1124}, $9.6$.
\end{rmk}

\begin{proof}[Proof of \ref{c:142}.]
Let $J \in \mk$ be an $\aleph_0$-saturated extension of $I$. Let $N = \GEM(J, \Phi)$. 
By hypothesis, there is a finite $\bar{t}_* \in \inc(I)$ and a function $F = F_{\bar{t}_*}$ giving a weak definition of $p$ over $\bar{t}_*$. 
Applying $F$ in the larger setting of $J$, let $q$ be the type 
$q = q_{ \bar{t}_*} (\bar{x})$ from Definition \ref{d:gd}.  Let 
\[ \mcs = \{ \bar{s} \subseteq J : \tpqf(\bar{s}, \emptyset, J) = \tpqf(\bar{t}_*, \emptyset, I) \}. \]
Recalling Observation \ref{o:corresp}, for each $\bar{s} \in \mcs$, 
let $F_{\bar{s}}$ denote the result of replacing $\bar{t}_*$ by $\bar{s}$ in the definition of $F$, 
and let $q_{\bar{s}} (\bar{x})$ denote the corresponding set of formulas. 
As $N$ is a $\GEM$-model, for each $\bar{s} \in \mcs$, 
$q_{\bar{s}} (\bar{x})$ is also a partial type. Let $N_1$ be a large elementary extension of $N$ in which each of the partial types in the set 
\[ \{  q_{\bar{s}} (\bar{x}) : \bar{s} \in \mcs \} \]
is realized, noting that $\bar{t}_* \in \mcs$ and therefore $q = q_{\bar{t}_*}$ belongs to this set.  Let $\bar{c}_{\bar{s}}$ denote a realization of 
$ q_{\bar{s}}$ in $N_1$. 
Let $\ma = \langle \bar{a}_t : t \in J \rangle$ denote the skeleton of the $\GEM$-model $N = \GEM(J, \Phi)$. 
Let $k = \lgn(\bar{x})$. 
Let $G_0, \dots, G_{k-1}$ be new $\lgn(\bar{a}_{\bar{t}_*})$-place function symbols.  
As $N \preceq N_1$, we may expand $N_1$ by interpreting the $G_i$'s so that 
\[ \langle G^{N_1}_i(\bar{a}_{\bar{s}}) : i < k \rangle = \bar{c}_{\bar{s}} \] 
for each $\bar{s} \in \mcs$. 
Finally, we may further expand $N_1$ by adding Skolem functions. 
Let $N^+_1$ denote this expanded model.  Apply the Ramsey property, Corollary \ref{d:ramsey-exp}, with $\ma$, $J$, $N^+_1$, $\Phi$ to obtain 
$\Psi \in \Upsilon[T]$, $\Psi \geq \Phi$. 

Why is this enough? By the reflection property mentioned in \ref{d:ramsey-exp}, the template $\Psi$ will record from $N^+_1$ the information 
that for each 
$\bar{s} \in \mcs$,  and each $(\bar{s}, \vp, \bar{\sigma}, \xr) \in \dom F_{\bar{s}}$ [where recall that $\xr \in \Dqf(I, \bar{s})$],  
\[  N^+_1 \models \vp ( ~  G^{N_1}_0(\bar{a}_{\bar{s}}), \dots, G^{N_1}_{k-1}(\bar{a}_{\bar{s}})~,~ \bar{s}(\bar{a}_{\bar{t}})~)^{F_{\bar{s}}(\bar{s}, \vp, \bar{\sigma}, \xr)} 
~~\mbox{ for every $\bar{t} \in \inc(J)$ realizing $\xr$.} \] 
That is, $\Psi$ records the truth or falsity of this formula as a property of $\tpqf(\bar{t}~^\smallfrown \bar{s}, \emptyset, J)$.  
This will ensure that in $\GEM(J, \Psi)$, for every $\bar{s} \in \mcs$, $\bar{G}(\bar{a}_{\bar{s}})$ 
will realize every formula of $q_{\bar{s}}$. This holds a fortiori in $\GEM(I, \Psi)$, which completes the proof. 
\end{proof}

\begin{claim} \label{c:143} 
Let $p$ be a partial type of $M = \GEM_{\tau(T)}(I, \Phi)$.  
Suppose there is some $\Psi \geq \Phi$ such that $p$ is realized in $\GEM(I, \Psi)$. 
Then $p$ has a weak definition over some finite $\bar{t}_* \in \inc(I)$ in $\GEM(I, \Psi)$. 
\end{claim}

\begin{proof} 
Suppose $\bar{c}$ realizes $p$ in $GEM(I,\Psi)$.  Let $\bar{t}_*$ be a finite subset of $I$ such that 
$\bar{c} \in \GEM(\bar{t}_*, \Psi)$, noting that if $\bar{c}$ is named by constants or is otherwise in the algebraic closure of the empty set, 
we may choose $\bar{t}_*$ to be empty.    Let $J \supseteq I$ be $\aleph_0$-saturated, and assume the skeleton $\ma$ of 
$\GEM(J, \Psi)$ extends that of $\GEM(I, \Psi)$. 

For any\footnote{assuming the given lengths of the variables, types, sequences are compatible.} 
tuple $(\bar{t}_*, \psi, \bar{\sigma}, \xr)$, let $F$ be given by $ F( \bar{t}_*, \psi, \bar{\sigma}, \xr ) = 1 $ when  
for some, equivalently every, finite $\bar{s} \subseteq I$ with $\tpqf(\bar{s}, \bar{t}_*, I) = \xr$, we have that 
$\GEM(I, \Psi) \models$ $\psi[\bar{c}, \bar{\sigma}(\bar{a}_{\bar{s}}) ]$;   and  $ F( \bar{t}_*, \psi, \bar{\sigma}, \xr ) = 0 $ otherwise. 
Since $\Psi$ is a template, and since $\bar{c}$ realizes $p$, this function is well defined and has the required properties.  
\end{proof}

\begin{disc}
\emph{If $p$ is a partial type of $M = \GEM_{\tau(T)}(I, \Phi)$, and has a weak definition in some $\GEM(I, \Psi)$ for $\Psi \geq \Phi$, 
must there be a weak definition already in $\GEM(I, \Phi)$? After all, $I$ has not changed.  This question has to do with the 
choice of detailed enumeration. If we fix a detailed enumeration of $p$ in $\GEM(I, \Phi)$, then whether or not this 
specific detailed enumeration gives rise to a weak definition is determined by $\Phi$;  a later, larger $\Psi \geq \Phi$ won't be able 
to change the situation.  However, our definition \ref{d:gd} starts by choosing in the given model, some detailed enumeration, 
and certainly with richer templates, the available detailed enumerations may increase. This is why a priori, a weak 
definition may become available later in some $\GEM(I, \Psi)$.}
\end{disc}

\begin{cor} Let $p$ be a partial type of $M = \GEM_{\tau(T)}(I, \Phi)$. \\  Then we have  (a) if and only if (b):
\begin{enumerate}
\item[(a)] there exists some $\Psi _1 \geq \Phi$ such that $p$ has a weak definition in $\GEM(I, \Psi_1)$ 
\\ $($and therefore 
has a weak definition in $\GEM(I, \Psi)$ for all $\Psi \geq \Psi_1$$)$.  
\item[(b)] there exists some $\Psi_2 \geq \Phi$ such that $p$ is realized in $\GEM(I, \Psi_2)$ 
\\ $($and therefore 
realized in $\GEM(I, \Psi)$ for all $\Psi \geq \Psi_2$$)$. 
\end{enumerate}
\end{cor}

\begin{proof}
If $p$ has a weak definition in $\GEM(I, \Psi_1)$ then Claim \ref{c:142} gives $\Psi_2 \geq \Psi_1$ such that 
$p$ is realized in $\GEM(I, \Psi_2)$. If $p$ is realized in $\GEM(I, \Psi_2)$ then letting $\Psi_1 = \Psi_2$, 
Claim \ref{c:143} shows that $p$ has a weak definition [that is, over some finite $\bar{t}_*$] in $\GEM(I, \Psi_1)$.  
For the parentheticals, being realized is clearly preserved under increasing the template by definition of the order on templates, 
and Observation \ref{o:larger} records that being weakly definable is too. 
\end{proof}

\begin{cor} Let $p$ be a partial type of $M = \GEM_{\tau(T)}(I, \Phi)$. 
\\  Then we have (a) if and only if (b):
\begin{enumerate}
\item[(a)] for no $\Psi _1 \geq \Phi$ does $p$ have a weak definition in $\GEM(I, \Psi_1)$.  
\item[(b)] for no $\Psi_2 \geq \Phi$ is $p$ is realized in $\GEM(I, \Psi_2)$. 
\end{enumerate}
\end{cor}

\vspace{5mm}

What is the core mechanism underlying the appearance of weak definability? 
\\ Suppose we look locally: this suggests:

\begin{qst} \label{first-q}
Let $p$ be a $\vp$-type or partial $\vp$-type in $M = \GEM_{\tau(T)}(I, \Phi)$.  Does there exist a finite sequence $\bar{t}_* \in \inc(I)$, 
a formula $\psi(x, \bar{b}) \in p$ $($so $\psi = \vp$ or $\neg \vp$$)$,  a finite sequence
$\bar{\sigma}$ of $\tau(\Phi)$-terms, and $\bar{t}  \in \inc(I)$ such that 
\begin{equation}
\label{fds0} \psi(x,\bar{b}) = \psi(x, \bar{\sigma}^{\GEM(I, \Phi)}(\bar{a}_{\bar{t}})) 
\end{equation}
and such that for some $\aleph_0$-saturated $J$, with $I \subseteq J \in \mk$,  in $N = \GEM(J, \Phi)$, the set of formulas 
\begin{equation}
\label{fds1} \{ \psi(\bar{x}, \sigma^N(\bar{a}_{\bar{s}})   ) : \tpqf(\bar{s}, \bar{t_*}, J) = \tpqf(\bar{t}, \bar{t}_*, I) \} 
\end{equation}  
\emph{extends $p$ and} is consistent? 
\end{qst}

Note that as $\Phi \leq \Psi$, and 
$\GEM(I, \Phi) \preceq \GEM(J, \Psi)$ by \ref{c:nice}, $\sigma(\bar{a}_{\bar{t}})$ evaluates identically in 
both the larger and smaller model.  Recall that the notation $J[\bar{t}_*]$ means $J$ expanded by constants for elements of $\bar{t}_*$. 

\begin{disc} \label{d-333}
\emph{Equations (\ref{fds0}) and (\ref{fds1}) tell us that in a possibly larger model there is a sequence 
\[ \langle \sigma^N(  \bar{a}_{\bar{s}})  : \bar{s} \in Q^{J[\bar{t_*}]}_\xr \rangle \]  
which is $\mk$-indiscernible over $\GEM(\bar{t}_*, \Psi)$ and which includes $\bar{b} = \sigma^N (\bar{a}_{\bar{t}})$, 
and the question essentially asks whether $\psi$ instantiated along this sequence is inconsistent.}   
\end{disc}

\br

The clarity brought by the larger $J$ is important, recalling e.g. \ref{d:sep}.
Still, the instructions as to whether or not to realize $p$ in $\GEM(I, \Phi)$ must come from the template $\Phi$. 
If inconsistency appears in $\GEM(J, \Phi)$,  this template cannot produce a realization even 
for $I$.

\br

\begin{disc}
\emph{In the next sections, we will see that this definition has a special explanatory power considered alone,  a priori free of 
connection to $\GEM$-models.} 
\end{disc}

\vspace{5mm}

\section{Shearing} \label{s:fd}
\setcounter{theoremcounter}{0}

In this section we develop a definition that will be central to the rest of the paper. Informally, it is the right extension of dividing (in the usual sense 
of model theory) 
to the case where we allow $\mk$-indiscernible sequences, for $\mk$ any index model class, not necessarily only linear orders.

\begin{ntn}  When $I_0 \subseteq J \in \mk$ is a set, writing  
$J[I_0]$ means $J$ expanded by constants for the elements of $I_0$, and likewise for $J[\bar{s}]$ when $\bar{s} \subseteq J$ is a sequence.
\end{ntn}

\begin{defn}[Shearing] \label{e5}  
Suppose we are given a context $\xc$, a theory $T$, $M \models T$, $A \subseteq M$, and a formula $\vp(\bar{x}, \bar{c})$ of the language of $T$ with 
$\bar{c} \in {^{\omega>} M}$. 

We say that 
\[ \mbox{ \emph{the formula $\vp(\bar{x}, \bar{c})$ shears over $A$ in $M$ for $(I_0, I_1, \xc)$} } \]
when there exist a model $N$, a sequence $\mb$ in $N$, enumerations $\bar{s}_0$ of $I_0$ and $\bar{t}$ of $I_1$, and an $\aleph_0$-saturated $J \supseteq I$ such that:
\begin{enumerate}
\item $I_0 \subseteq I_1$ are finite subsets of $I$
\item $M \preceq N$
\item $\mb = \langle \bar{b}_{\bar{s}} : \bar{s} \in {^{\omega >}(J[I_0])} \rangle $ is $\mk$-indiscernible in $N$ over $A$
\item $\bar{c} = \bar{b}_{\bar{t}}$, and 
\item the set of formulas 
\[ \{ \vp(\bar{x}, \bar{b}_{\bar{t}^\prime}) : \bar{t}^\prime \in {^{\lgn(\bar{t})}(J)}, 
\tpqf(\bar{t}^\prime, \bar{s}_0, J) = \tpqf(\bar{t}, \bar{s}_0, I) \} \]
is contradictory.
\end{enumerate}
\end{defn}

\begin{conv} \emph{Some conventions for Definition \ref{e5}:}
\begin{enumerate}
\item \emph{If $\xc$ is clear, we may write ``...for $(I_0, I_1)$'' instead of ``...for $(I_0, I_1, \xc)$.''}
\item \emph{We may write ``the formula $\vp(\bar{x}, \bar{c})$ shears over $A$ in $M$'' to mean that there is some $(I_0, I_1)$ 
for which this holds.} 
\end{enumerate}
\end{conv}

\begin{obs} 
Changing $J$ in the definition \ref{e5} does not matter as long as $I_1 \subseteq J \in \mk$ and $J$ is $\aleph_0$-saturated.
\end{obs}

\begin{disc} \emph{Definition \ref{e5} is parallel to the usual notion of dividing in that 
$M$ is not assumed to be a $\GEM$-model. Even if it is, $\mb$ need have no connection to the skeleton. 
Only the parameter $\bar{c}$ of $\vp(\bar{x},\bar{c})$ and the set $A$ are required to be in $M$  (though by the first of these 
$\mb$ has nonempty intersection with $M$). 
A priori, the sequence $\mb$ belongs to $N$. }
\end{disc}

\begin{claim} \label{c:monot}
Suppose $\vp(\bar{x}, \bar{c})$ shears over $A$ in $M$ for $(I_0, I_1, \xc)$. Suppose $I^\prime_1 \supseteq I^\prime_0 \subseteq I_0$ and 
$I_1 \subseteq I^\prime_1$. Then $\vp(\bar{x}, \bar{c})$ shears over $A$ in $M$ for $(I^\prime_0, I^\prime_1, \xc)$. 
\end{claim}

\begin{disc}
\emph{Suppose $\vp(\bar{x}, \bar{c})$ shears over $A$ in $M$ for $(I_0, I_1, \xc)$, as witnessed by $\bar{s}_0, \bar{t}$, $N$, and $\mb$.  
The same data work to show that $\vp(\bar{x}, \bar{c})$ shears in elementary extensions of $N$, and also if we take the reduct of $N$ 
to a language which still contains $\tau(\vp)$.  However, shearing does {not} necessarily persist under expansions 
(consider what happens if we name $\bar{c}$ by a constant.)}
\end{disc}

\vspace{4mm}

\begin{claim}[Dividing implies shearing] \label{needed1}
Let $T$ be any complete theory and 
suppose $\vp(\bar{x}, \bar{a})$ divides over some set $A$ in the monster model of $T$.  Then $\vp(\bar{x}, \bar{a})$ $\xc$-shears over $A$ for any 
context $\xc$. 
\end{claim}

\begin{rmk} 
The proof will show more, namely that we can choose any finite $I_0 \subseteq I$ and any $I_1 \supseteq I_0$ such that 
$I_1 = I_0 \cup \{ t \}$ where $t \notin \cl(I_0)$, and $\vp$ will $(I_0, I_1, \xc)$-shear over $A$.  
$($In fact there is nothing in the proof that prevents $t$ from having length longer than $1$.$)$ 
\end{rmk}

\begin{proof}[Proof of \ref{needed1}] 
The idea of the proof is simple: use the Ramsey property to upgrade a dividing sequence to a sequence witnessing shearing. 
However, we check all the details.

By our assumption, there are $1< k<\omega$ and a formula 
\begin{equation} \label{e:param}
\vp(\bar{x}, \bar{a})
\end{equation}
which $k$-divides over $A$ in the monster model of $T$.  
Let $\xc = (I, \mk) = (I_\xc, \mk_\xc)$ be given, recalling that this means satisfying 
\ref{d:imc} and \ref{d:context}. 
Fix any finite $I_0 \subseteq I$ and choose $t \in I_\xc \setminus I_0$ with 
$t$ not in the definable closure of $I_0$ in $I_\xc$.
Fix an enumeration $\bar{s}_0$ of $I_0$. 
 Let $J$ be $\aleph_0$-saturated, and without loss of generality, $I_0 \subseteq J$. 
Let 
\[  \mcs = \{ s \in J : \tpqf(t, \bar{s}_0, J) = \tpqf(s, \bar{s}_0, I) \} \]
which is infinite by our assumption \ref{d:context}, and inherits a linear ordering from $J$.   Let $I_1 = I_0 \cup \{ t \}$. We will show that 
$\vp$ shears for $(I_0, I_1, \xc)$. 
Let 
\[ \mb = \langle \bar{b}_s : s \in \mcs \rangle, \mbox{ i.e. }  \langle \bar{b}_s : s \in (\mcs, <) \rangle \] 
[where the intention is: indexed by $\mcs$ considered as a linearly ordered set] be an indiscernible sequence over $A$ in the monster model of 
$T$ witnessing the $k$-dividing of $\vp$. To belabor the point, $\mb$ is only indiscernible in the usual sense, as by compactness we can choose such a sequence indexed by any infinite linear order. In particular, 
$\bar{a}$ from (\ref{e:param}) belongs to $\mb$, and 
\[ \{ ~\vp(\bar{x}, \bar{b}_s) : s \in \mcs ~\}  \]
is $1$-consistent but $k$-contradictory. 

We now appeal to a $\GEM$-model. Choose some template $\Psi$ proper for $\mk$ so that 
$M \rstr \tau(T) \models T$, where $M = \GEM(J, \Phi)$ (this is always possible, see e.g. \ref{cont-for-t}). Let $\ma = \langle \bar{a}_s : s \in J \rangle$ be the skeleton of $M$. 
Let $M_0$ be an elementary extension of $M$ which contains $\mb$ and $A$. Let $M^+_0$ be the expansion of $M_0$ in which every element of $A$ 
is named by a constant. 
Let $m_0 = |I_0 \cup \{ t \}|$ and $m_1 = \lgn(\bar{a}_s)$ where $\bar{a}_s$ is any member of the skeleton, and 
let $\ell = \lgn(\bar{b}_s)$ for any $\bar{b}_s \in \mb$.  
Let $m = m_0 \cdot m_1$.  [Even though we are looking to build an $\mcs$-indexed sequence, 
recall we are working over $I_0$, so we need to carry out the next 
part of the construction uniformly over all copies of $I_0$ for the Ramsey property to work as desired.] 
Let $M^{++}_0$ be the following further expansion of $M^+_0$: add to the language a new sequence of $\ell$-many 
$m$-ary function symbols. Interpret these functions in $M^{++}_0$ so that 
for each $\bar{s}^\prime_0 \in {^{\omega>}J}$ with $\tpqf(\bar{s}^\prime_0, \emptyset, J) = \tpqf(\bar{s}_0, \emptyset, I)$,  and for  
each $s \in \mcs$, 
\[ F^{M^{++}_0}_i (\bar{a}_{\bar{s}^\prime _0~^\smallfrown \langle s \rangle}) : i < \ell \rangle = \bar{b}_s. \]
Apply the Ramsey property to $M, \ma, M^{++}_0, \Phi$ and let $\Psi \geq \Phi$ be the template returned. 

Let $N = \GEM(J, \Psi)$. 
Note that in $N \rstr \tau(T)$ there will be an automorphic image $A^\prime$ of $A$, which is 
named by constants in the $\GEM$-model $N$. 
For each $s \in \mcs$, let 
\[ \bar{b}^\prime_s = \langle F^{N}_i (\bar{a}_{\bar{s}_0 ~ ^\smallfrown \langle s \rangle}) : i <\ell \rangle. \]
To match the notation of Definition \ref{e5}: for some, equivalently every, $s \in \mcs$, let $\xr = \tpqf(s, \emptyset, J[\bar{s}_0]) 
=\tpqf(t, \emptyset, J[\bar{s}_0])$.   
Recall that $\{ s \in Q^{J[\bar{s}_0]}_\xr \}$ denotes $\{ s \in J, \tpqf(s, \emptyset, J[\bar{s}_0]) = \xr \}$. 
Then by the Ramsey property, 
\[ \mb^\prime = \langle \bar{b}^\prime_s : s \in Q^{J[\bar{s}_0]}_\xr \rangle \]
is an $\xr$-indiscernible sequence, and [because $A^\prime$ is named by constants in 
$N$] it is $\xr$-indiscernible over $A^\prime$, the copy of $A$ in $N$. 
Moreover, by the Ramsey property and the choice of our original sequence $\mb$, 
\[ \{ \vp(x,\bar{b}^\prime_s) : s \in Q^{J[\bar{s}_0]}_\xr \} \] will be $1$-consistent but 
$k$-inconsistent. 
Finally, observe that we may without loss of generality assume $\bar{a}$ from (\ref{e:param})  
belongs to the sequence $\mb^\prime$, as follows. By the Ramsey property, 
the type of $A, \bar{a}$ in the monster model for $T$ will be the same as 
that of $A^\prime, \bar{b}^\prime_t$.  So we may move $A^\prime$ to $A$ by an automorphism $G$, 
and then move $G(\bar{b}^\prime_t)$ to $\bar{a}$ by an automorphism $F$ which fixes $A$ pointwise. 
Then the sequence $F(G(\mb^\prime))$ will witness the $(I_0, I_1, \xc)$-shearing of $\vp$ as desired.  
\end{proof}

For later quotation, we single out the special case of linear orders. 

\begin{cor} \label{needed1a}
Let $M$ be any model, $A \subseteq M$ 
and let $\vp(\bar{x}, \bar{a})$ be any formula of $M$. Suppose $\vp(\bar{x}, \bar{a})$ divides over $A$ in the 
usual sense. Let $\mk$ be the class of linear orders, $I$ any infinite member of $\mk$, and $\xc = (I, \mk)$. 
Then for any finite 
$I_0 \subseteq I$ and any $t \in I \setminus I_0$, writing $I_1 = I_0 \cup \{ t \}$, 
we have that $\vp(\bar{x}, \bar{a})$ ~  
$(I_0, I_1, \xc)$-shears. 
\end{cor}

\vspace{4mm}

\subsection*{Discussion and examples}

We outline here some additional results (originally part of the present paper, but moved to a separate manuscript for reasons of 
space) to give the reader some idea of the landscape. 

The definition of ``$\xc$-unsuperstable'' is given in the next section, and for now can be understood as a strong version of ``shearing occurs''. 

Recall that $T_{3,2}$ is the theory of the generic tetrahedron-free three-hypergraph.
(We have kept the notation consistent with our earlier papers. In his work Ulrich has suggested a reasonable notational change, adding one to the subscripts.) 

\begin{fact} \label{fact15}
For each $n \geq 2$, let $T_{n,1}$ be the theory of the generic $K_{n+1}$-free graph.  Then $T_{n,1}$ is not simple.  
\end{fact}

\begin{fact}[Hrushovski c. 2002, see \cite{hrushovski1}]  \label{t:trd}
For $n>k\geq 2$, $T_{n,k}$ is simple unstable with only trivial dividing $($i.e. only dividing coming from equality$)$. 
\end{fact}

The following is worked out in \cite{MiSh:F2061}: 

\begin{expl}[\cite{MiSh:F2061}] \label{ex:t32}
There is a countable context $\xc$ such that the random graph is $\xc$-superstable, and 
$T_{3,2}$ is $\xc$-unsuperstable;  in $T_{3,2}$, this shearing arises from a formula which is a Boolean combination of 
positive instances of the edge relation.  
\end{expl}

The following general class of examples are constructed in \cite{MiSh:F2061}: 

\begin{expl}[\cite{MiSh:F2061}] \label{unary} Given any $n> k\geq 2$, there is a countable context $\xc$ 
such that $\trg$ is $\xc$-superstable and $T_{n,k}$ is not $\xc$-superstable.  
\end{expl}

\begin{concl} \label{e:t32}
Shearing is strictly weaker than dividing.  
\end{concl}

\begin{disc} \label{d:indr}
\emph{Since shearing is not the same as dividing in simple theories, it necessarily fails some of the usual properties of 
independence relations.}
\end{disc}

\vspace{5mm}

\section{Unsuperstability}  \label{s:unss}
\setcounter{theoremcounter}{0}

\begin{hyp} \label{m2x} \emph{ } 
\begin{enumerate}[1)]
\item $\xc$  is a context, so $I$ and $\mk$ are fixed. 
\item In this section \underline{$I$ is countable}.  We may say: $\xc$ is a countable context. 
\item $J$ is $\aleph_0$-saturated, $I \subseteq J \in \mk$.
\item $T$ will vary, but will always be a complete first order theory.  
\item $\Upsilon_\xc$ denotes the templates proper for $\mk$. We will assume the templates 
$\Phi$ in question satisfy $T_\Phi \supseteq T$ and have Skolem functions for $T$,  i.e. belong to $\Upsilon_\xc[T]$. 
\item Note: when we write $\tpqf(\bar{s}, ...,J) = \tpqf(\bar{t}, ..., J)$ or something of the sort, it's understood 
that $\lgn(\bar{s}) = \lgn(\bar{t})$.  
\end{enumerate}
\end{hyp}

In this section we define ``$T$ is (un)superstable for the countable context $\xc$'' and prove Theorem \ref{t:dir2}.\footnote{We might have said 
``unsupersimple.''} To do so 
we step back from our assumption that $I$ must be the index set for the skeleton of a given $\GEM$ model, to 
simply using $I$ (or a saturated $J$ extending it) as the index set for \emph{some} $\mk$-indiscernible sequence which 
will witness e.g. inconsistency or dividing.  Notice that in the next definition, $I$ is not a priori an input 
to a GEM model, and the $B_n$ are just sets in the monster model, a priori not related to the $I_n$'s 
beyond what is written there. 

\begin{defn} \label{m5a}  Let $\xc$ be a countable context. 
We say $T$ is \emph{unsuperstable} for $\xc$ when there are: 

\begin{enumerate}
\item[(a)] 
an increasing sequence of nonempty finite sets $\langle I_n : n <\omega \rangle$ with 
$I_m \subseteq I_{n} \subseteq I$ for $m < n <\omega$ and $\bigcup_n I_n = I$, 
which are given along with a choice of enumeration $\bar{s}_n$ for each $I_n$ 
where $\bar{s}_n \tlf \bar{s}_{n+1}$ for each $n$ 

\item[(b)]  
an increasing sequence of nonempty, possibly infinite, sets $B_n \subseteq B_{n+1} \subseteq \mathfrak{C}_T$ in the monster 
model for $T$, with $B := \bigcup_n B_n$

\item[(c)] and a partial type  $p$ over $B$, such that  

\end{enumerate}
\[ p \rstr B_{n+1}~~ (I_n, I_{n+1})\mbox{-shears over }B_n. \]
\end{defn}

\begin{rmk}
To extend this definition to $I$ each of whose strict subsets is finitely generated,  add that $\cl(I_n) \subsetneq \cl(I_{n+1})$
and make the parallel changes to the proofs so that the $I_n$'s list finite sequences of generators rather than their closures. 
\end{rmk}

\begin{defn}
When $T$ is not unsuperstable for $\xc$, we say \emph{$T$ is superstable for $\xc$}, or just \emph{$\xc$-superstable}. 
\end{defn}

\begin{rmk}
\emph{Definition \ref{m5a} uses countability of $I$ in an essential way, as it is the union of an increasing chain of finite sets.}
\end{rmk}

For later reference, we state the local version separately.  Comparing to \ref{m5a}, note 
``$T$ is superstable for the countable context $\xc$'' is just the case where 
$(T, \Delta)$ is $\xc$-superstable and $\Delta$ is the set of all formulas of the language.

\begin{defn} \label{m5aa}  Let $\xc$ be a countable context and $\Delta$ a set of formulas of $T$.  
We say 
\[ \mbox{ $(T, \Delta)$ is \emph{unsuperstable} for $\xc$ } \]when $\ref{m5a}$ holds in the case that we replace 
``$\ts(B)$'' in $\ref{m5a}(c)$ by ``$\ts_\Delta(B)$,'' i.e., the type $p$ in $\ref{m5a}(c)$ 
may be taken to be a $\Delta$-type. 
\end{defn}

\begin{claim} 
\label{needed2}
Suppose $(T, \Delta)$ is not supersimple in the usual sense and $\xc = (I, \mk)$ is any countable context. 
Then $(T, \Delta)$ is unsuperstable for $\xc$. 
\end{claim}

\begin{proof} 
Immediate from \ref{needed1}; the countability of the context is used only in the definition of $\xc$-superstable. 
\end{proof}

For the complementary claim, see \ref{needed3} below. 

\begin{obs} \label{cont-for-t}
For any theory $T$ and context $\xc$, for any $\Phi \in \Upsilon_\xc$, there is  $\Psi$ with $\Phi \leq \Psi \in \Upsilon_\xc[T]$.  
\end{obs}

\begin{proof}
Let $J \supseteq I$ be $\aleph_0$-saturated. 
Let $M = \GEM(J, \Phi)$ with skeleton $\ma$. If $\Phi$ is not already in $\Upsilon_\xc[T]$, then without loss of generality, we may suppose the 
signature of $T$ and of $T_\Phi$ are disjoint. Let $N$ be an elementary extension of $M$ which may also 
be expanded to a model of $T$ with Skolem functions for $T$ (of course this expansion need not have anything to do with 
the structure on $N$). 
Let $N^+$ be this expansion.  Let $\Psi$ be the template returned by applying the Ramsey property to $N^+$, $\ma$, $\Phi$. 
Then $\Psi  \geq \Phi$ and $\Psi$ will be in $\Upsilon_\xc[T]$. 
\end{proof}

First we consider the case where $T$ is superstable for a given context, i.e., 
not unsuperstable. The larger role of $I$ mentioned above plays little role in this proof, since superstability 
ensures good behavior for \emph{all} relevant $I_n$'s and $B_n$'s, including those which have natural meaning in a 
$\GEM$ model.

\begin{claim} \label{m20aa}
Assume $T$ is superstable for the countable context $\xc$. Suppose ${\Phi \in \Upsilon_\xc[T]}$ and 
$M = \GEM_{\tau(T)}(I, \Phi)$.  Let $p \in \ts(M)$ be any type. Then there is $\Psi \geq \Phi$ such that 
$p$ is realized in $\GEM(I, \Psi)$. 
\end{claim}

\begin{proof} 
Let $\langle i_n : n <\omega \rangle$ list $I$. Let $I_n = \{ i_k : k < n \}$. Let 
$B_n$ name the model $\GEM_{\tau(T)}(I_n, \Phi)$. Note that $\bigcup_n I_n = I$, 
$\bigcup_n B_n = M$, and $p \in \ts(M)$. 
We ask: is there $m < \omega$ such that for no $n > m$ does 
$p \rstr B_{n}$ $(I_m, I_{n})$-shears over $M_m$? 

If there is no such $m$, so we contradict superstability. 
More precisely, 
choose $n(i)$ increasing with $i$ such that $i = j+1$ implies 
$p \rstr B_{n(i)}$ ~ $(I_{n(j)}, I_{n(i)})$-shears over $B_{n(j)}$. Now the sequences 
$\langle I_{n(i)} : i < \omega \rangle$, $\langle B_{n(i)} : i < \omega \rangle$, and the type 
$p \in \ts_{\tau(T)}(B) = \ts_{\tau(T)}(\bigcup_i B_{n(i)} )$ witness that $T$ is $\xc$-unsuperstable.

So there must be one such, call it $m_*$. Now we proceed similarly to the case where we have a weak definition. 
Let $J \supseteq I$ be $\aleph_0$-saturated. Since $M$ is a $\GEM$-model, we may choose a detailed enumeration 
(recalling \ref{d:detailed}) 
\[ p(\bar{x}) = \{ \vp_\alpha(\bar{x}, \bar{\sigma}_\alpha(\bar{a}_{\bar{t}_{\alpha}} )) : \alpha < \alpha_*\} \] 
where each $\bar{a}_{\bar{t}}$ is from the skeleton and each $\bar{\sigma}_\alpha$ is sequence of $\tau(\Phi)$-terms.  
Let $\bar{s}_{m_*}$ be the enumeration of $I_{m_*}$. 
Consider the larger set of formulas 
\begin{equation}
\label{eq:q} q(\bar{x}) = \{ \vp_\alpha(\bar{x}, \bar{\sigma}_\alpha(\bar{a}_{\bar{t}})): \alpha < \alpha_*, 
~ \bar{t} \in {^{\omega >} J}, 
 ~\tpqf(\bar{t}, \bar{s}_{m_*} J) = 
\tpqf(\bar{t}_{\alpha}, \bar{s}_{m_*}, I) \}. 
\end{equation}
Suppose $q(\bar{x})$ were not a partial type.  There would be $\alpha_1, \dots, \alpha_k$ such that 
\[ \{ \vp_\alpha(\bar{x}, \bar{\sigma}_\alpha(\bar{a}_{\bar{t}})): \alpha  \in \{ \alpha_1, \dots, \alpha_k \}, 
\bar{t} \in {^{\omega >} J}, 
 \tpqf(\bar{t}, \bar{s}_{m_*} J) = 
\tpqf(\bar{t}_{\alpha}, \bar{s}_{m_*}, I) \} \] 
is inconsistent.  Assuming the model $M$ is infinite (if not it would already be saturated), without loss of 
generality\footnote{More precisely, there is some $\beta$: let $\vp_\beta$ be the conjunction of the formulas 
$\vp_\alpha$. Then $\vp_\beta$ belongs to the type so is consistent, but if we allow the relevant $\bar{t}$ to vary in $J$, 
we get inconsistency by definition of $\beta$. This is the only point in the proof where we use that $p$ is a 
type, that is, that $p$ is a $\Delta$-type where $\Delta$ is a set of formulas closed under conjuction. However, notice that if 
$p$ is a $\vp$-type, we can always choose $\vp_\beta$ to be (an instance of $\vp$) $\land$ (an instance of $\neg \vp$).}
there is some single $\alpha$ such that 
\[ \{ \vp_\alpha(\bar{x}, \bar{\sigma}_\alpha(\bar{a}_{\bar{t}})):  \bar{t} \in {^{\omega >} J}, 
 \tpqf(\bar{t}, \bar{s}_{m_*} J) = 
\tpqf(\bar{t}_{\alpha}, \bar{s}_{m_*}, I) \} \] 
is contradictory. 
Now the  sequence
\[ \langle \bar{\sigma}_\alpha(\bar{a}_{\bar{t}}))  :  
\bar{t} \in {^{\omega >} J [\bar{s}_{m_*}]} \rangle \] 
is $\mk$-indiscernible (the intended interpretation is that when $\bar{a}_{\bar{t}}$ has the wrong length to input to $\bar{\sigma}_\alpha$, 
the expression evaluates to $\emptyset$) over $\bar{a}_{\bar{s}_{m_*}}$. In other words, it is $\mk$-indiscernible over $B_{m_*}$.   
Let $n$ be such that $\bar{t}_\alpha \subseteq I_n$. 
Then we've shown that the formula $\vp_\alpha(\bar{x}, \bar{\sigma}_\alpha(\bar{a}_{\bar{t}_\alpha}))$ here 
$(I_{m_*}, I_n)$-shears over $B_{m_*}$. This contradicts the choice of $m_*$ from the beginning of the proof. 
We conclude that $q(\bar{x})$ is indeed a partial type, and of course $q(\bar{x}) \supseteq p(\bar{x})$. 

Now, for any other $\bar{s} \in {^{\omega >}J}$ such that $\tpqf(\bar{s}, \emptyset, J) = \tpqf(\bar{s}_{m_*}, \emptyset, I)$, 
let 
\[ q_{\bar{s}}(\bar{x}) \] 
denote the result of replacing $\bar{s}_{m_*}$ by $\bar{s}$ in $(\ref{eq:q})$ above.  This takes place in  
$\GEM(J, \Phi)$, and $q(\bar{x}) = q_{\bar{s}_{m_*}}(\bar{x})$ is a partial type, so each $q_{\bar{s}}(\bar{x})$ must also be a partial type.  

Thus, in some larger elementary extension $N_{\star}$ of $\GEM(J, \Phi)$, we may realize all of these types $q_{\bar{s}}(\bar{x})$. Let 
$\bar{d}_{\bar{s}}$ denote the realization in $N_{\star}$ of $q_{\bar{s}}(\bar{x})$. 
Expand $N_{\star}$ to $N^+_{\star}$ by new functions $F_\ell$, $\ell < \lgn(\bar{x})$, interpreted so that 
for each $\bar{s} \subseteq J$ realizing $\tpqf(\bar{s}_{m_*}, \emptyset, J)$, we have 
\[ \langle F_\ell(  \bar{a}_{        \bar{s}     }  )  : \ell \rangle = 
\bar{d}_{ \bar{s} }. \]   
Finally, let $N^{++}_{\star}$ be the expansion of 
$N^+_{\star}$ to a model with Skolem functions. Applying the Ramsey property with $\GEM(J, \Phi)$, $\ma$, and 
$N^{++}_{\star}$, let $\Psi$ be the template returned. Then $\Psi$ will be nice, proper for $\mk$, and in 
$\GEM(I, \Psi)$ the type $p$ will be realized, as will be many of its copies. 
\end{proof}

\begin{cor} \label{m20ab}
Assume $(T,\Delta)$ is superstable for the countable context $\xc$. Suppose ${\Phi \in \Upsilon_\xc[T]}$ and 
$M = \GEM_{\tau(T)}(I, \Phi)$.  Let $p \in \ts_\Delta(M)$ be any type. Then there is $\Phi \geq \Psi$ such that 
$p$ is realized in $\GEM(I, \Psi)$. 
\end{cor}

\begin{proof}
The same proof works at a slight notational cost; simply replace $\ts$ by $\ts_\Delta$, and add $\Delta$ to $\ts_{\tau(T)}(B)$.
\end{proof}

\begin{claim} \label{m20a}
Let $\Delta$ be any set of formulas of $T$, in our main case all formulas. 
Assume $(T, \Delta)$ is superstable for the countable context $\xc$. 
Let $\mu$ and $\lambda$ be such that $\mu > |T|$,  $\lambda = \lambda^{<\mu}$.
Then for a dense set of $\Psi \in \Upsilon_\xc[\lambda, T]$ 
the model $\GEM_{\tau(T)}(I, \Psi)$ is $\mu$-saturated for $\Delta$-types. 
\end{claim}

\begin{proof} 
Choose $\Phi_0 \in \Upsilon_\xc[\lambda, T]$, recalling this denotes the templates $\Psi$ proper for $\mk_\xc$ 
with $|\tau(\Psi)| \leq \lambda$ and $T_\Psi \supseteq T$.  We need to show that for any such 
$\Phi_0$ there is $\Psi \geq \Phi_0$ as required. 

By induction on $\alpha \leq \lambda$ we will construct an increasing continuous chain of templates $\Phi_\alpha \in \Upsilon_\xc[\lambda, T]$ so that $\Phi_\lambda$ 
will have the desired property.  It suffices to describe the successor stage. 
Let $M_\alpha = \GEM_{\tau(T)}(I, \Phi_\alpha)$. Since $\tau(\Phi_\alpha)$ has $\leq \lambda$ symbols, this will be a model of size $\leq \lambda$. 
Counting types, there will be $\lambda = \lambda^{<\mu}$ choices of a parameter set $A$ of size $<\mu$, and over 
each such $A$, up to $2^{<\mu} \leq \lambda$ types, for a total of $\leq \lambda$ types.  Applying  
Claim \ref{m20a} (either applying that Claim $\lambda$ times in succession, or better, simply modifying that proof by adding $\lambda$-many different 
functions $F$ and realizing the types all at once), we find $\Phi_{\alpha + 1} \geq \Phi_\alpha $ with 
$|\tau(\Phi)_{\alpha+1}| \leq |\tau(\Phi_\alpha)| + \lambda$ so that in $\GEM_{\tau(T)}(I, \Phi_{\alpha+1})$ the types we had 
just counted are all realized.  

By the end of the induction, $M_\lambda = \GEM_{\tau(T)}(I, \Phi_\lambda)$ will be $\mu$-saturated. 
\end{proof}

\noindent A comment on the operation of Claim \ref{m20a}.  
At first it may seem strange that saturated models are built up around a single unchanging $I$, 
but what one should notice is the change and expansion in the template as $\Phi_0$ becomes $\Phi_\lambda$. 
In some sense the induction of \ref{m20a} is simply adding a growing list of precise construction instructions to the `scaffolding' 
of the model (the saturation will be for $\tau(T)$ once the `scaffolding is taken off'). 
The inclusion of both $\mu$ and $\lambda$ in the statement of the claim points out how we 
may increase saturation even further as we allow an increased distance between the size of $\tau(\Phi)$ and 
the ``constant'' size of $T$.  If we hope to build a $\mu$-saturated model for some large $\mu$, 
the statement of Claim \ref{m20a} tells us what kind of $\lambda$ we will need.

\br

Next we consider $\xc$-unsuperstability. 
In this direction, the potential difference between the $I$- or $J$-indexed sequence witnessing shearing and 
the $I$- or $J$-indexed skeleton of the $\GEM$ models in the picture will be noted. 

\begin{disc} Our theorems will continue to be true locally as will be obvious from the proofs (the type ultimately omitted is a progressive 
automorphic image of the type realizing un-superstability), 
but we emphasize the global versions as there is marginally less notation, and state 
the local versions after for later reference. 
\end{disc} 

As a warm-up for Theorem \ref{t:dir2}, we explain how to copy a single instance of shearing into a $\GEM$-model. 
(Claim \ref{c:folding} is illustrated in Figure 1.)

\begin{figure}
\begin{center}
\includegraphics[width=120mm]{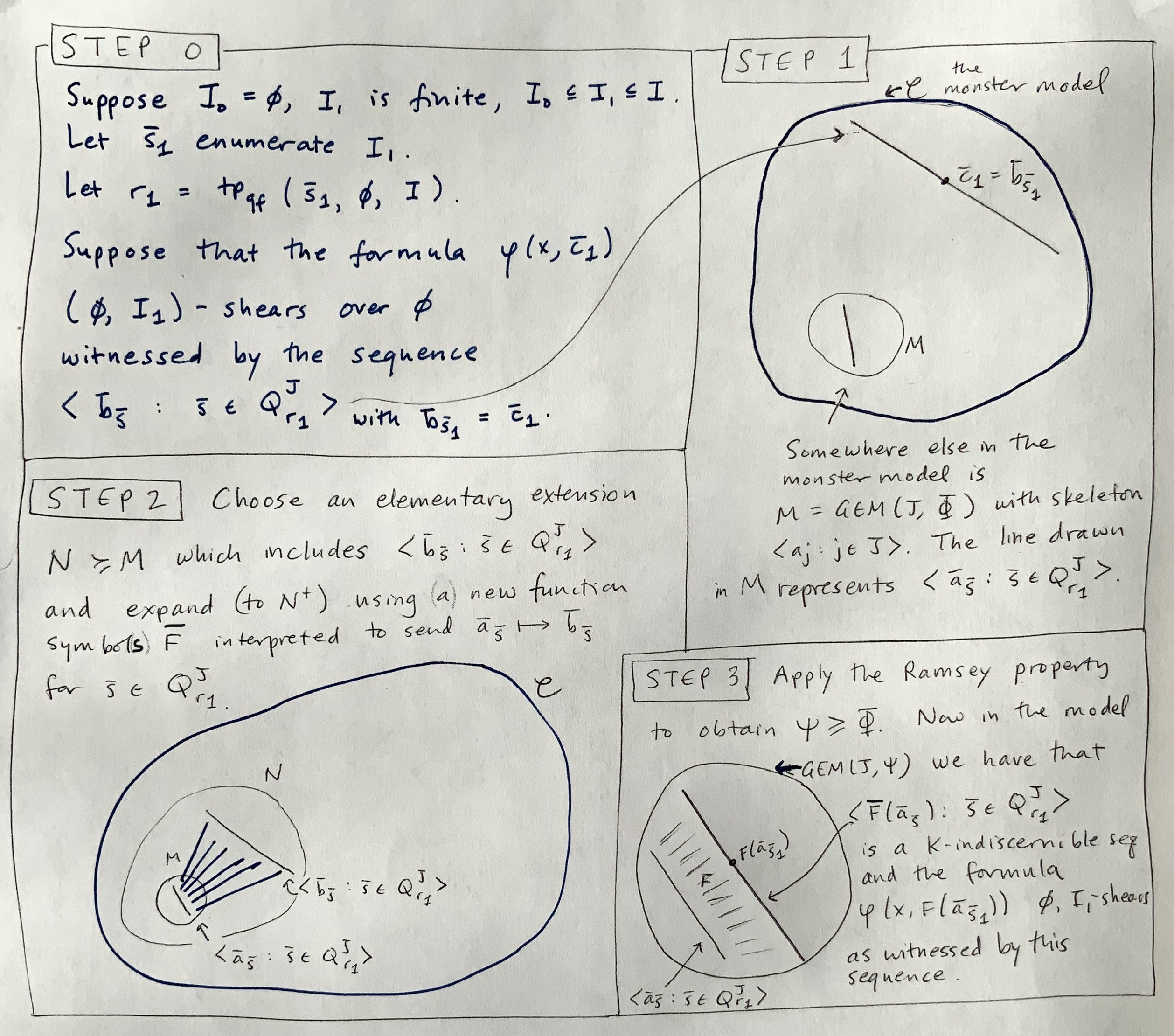}
\end{center}
\caption{Folding an instance of shearing into a $\GEM$-model (\ref{c:folding}).}
\end{figure}

\begin{claim}[Folding an instance of shearing into a $\GEM$-model] \label{c:folding} 
Fix a background theory $T$ and suppose we are given a countable context $\xc = (I, \mk)$, along with:
\begin{enumerate}
\item \emph{(an instance of shearing)} Let $I_0 = \emptyset$ and $I_1 \subseteq I$ be finite. 
Let $J \supseteq I$, $J \in \mk$ be $\aleph_0$-saturated. Let $\bar{s}_1$ enumerate $I_1$. Let 
$\xr_1 = \tpqf(\bar{s}_1, \emptyset, I)$. Suppose the formula $\vp(x,\bar{c}_1)$ shears for $(\emptyset, I_1)$, 
witnessed by the $\mk$-indiscernible sequence $\mb = \langle \bar{b}_{\bar{s}} : \bar{s} \in Q^J_{\xr_1} \rangle$, with 
$\bar{b}_{\bar{s}_1} = \bar{c}_1$.  Let $m = \lgn(\bar{c}_1)$. 

\item \emph{(a $\GEM$-model)} Let $M = \GEM(J, \Phi)$ where $\Phi \in \Upsilon[T]$.  Let 
$\langle \bar{a}_j : j \in J \rangle$ be the skeleton of $M$, and let $k = \lgn(\bar{a}_{\bar{s}_1})$. 
\end{enumerate}
Then there exists $\Psi \geq \Phi$ and a sequence of $k$-place function symbols $\langle \bar{F}_\ell : \ell < m \rangle$ 
of $\tau(\Psi) \setminus \tau(\Phi)$ such that in the model $\GEM(J, \Psi)$ 
which has skeleton $\langle \bar{a}_j : j \in J \rangle$, we have that 
the formula $\vp(x, \bar{F}(\bar{a}_{\bar{s}_1}))$ itself $(\emptyset, I_0)$-shears, witnessed by the sequence 
$\langle \bar{F}(\bar{a}_{\bar{s}}) : \bar{s} \in Q^J_{\xr_1} \rangle$.  Moreover the type of this sequence is 
the same as the type of $\mb$. 
\end{claim}

\begin{proof}
Consider $M \rstr \tau(T)$ and $\mb$ inside the monster model, having a priori nothing to do with each other. 
(We are initially in $\mathfrak{C}_T$, though when we expand before applying the Ramsey property, it is really $\mathfrak{C}_{T_\Phi}$. So the reader should understand that $\mathfrak{C}_T$ denotes ``$\mathfrak{C}_{T_\Phi} \rstr \tau(T)$.'') 
Choose an elementary extension $N$ of $M$ which includes $\langle \bar{b}_{\bar{s}} : \bar{s} \in Q^J_{\xr_1} \rangle$. 
Let $\langle F_\ell : \ell < m \rangle$ be new $k$-place function symbols, where new means not in $\tau(\Phi)$. 
Expand $N$ to $N^+$ by interpreting the functions $F$ so that 
\[ \bar{F}(\bar{a}_{\bar{s}}) = \bar{b}_{\bar{s}} \mbox{ for each $\bar{s} \in Q^J_{\xr_1}$. } \]
Recall our notation: if $\bar{F} = \langle F_i : i < k \rangle$ then $\bar{F}(\bar{a})$ is the sequence 
$\langle F_i(\bar{a}) : i < k \rangle$. Expand also to add Skolem functions. Apply the Ramsey property to $N^+$ to obtain $\Psi \geq \Phi$.  
Now in $\GEM(J, \Psi)$, with skeleton $\langle \bar{a}_j : j \in J \rangle$, we have that 
\[  \{ \bar{F}(\bar{a}_{\bar{s}}) : \bar{s} \in Q^J_{\xr_1} \} \]
is a $\mk$-indiscernible sequence, necessarily contained in $\GEM(J, \Psi)$, and it witnesses the 
$(\emptyset, I_1)$-shearing of the formula $\vp(x,\bar{F}(\bar{a}_{\bar{s}_1})$, as desired. 
The last line of the claim follows by the Ramsey property, \emph{because} $\mb$ was a $\mk$-indiscernible sequence. 
\end{proof}

\begin{theorem} \label{t:dir2}
Suppose $\xc$ is a countable context and assume $T$ is $\xc$-unsuperstable.  
For every $\Phi \in \Upsilon_\xc[T]$ there is $\Phi_* \geq \Phi$ with $|\tau(\Phi_*)| \leq |T| + |\tau_\Phi| + \aleph_0$ such that 
for every $\Psi \geq \Phi_*$ we have  
$\GEM_{\tau(T)}(I, \Psi)$ is not $\aleph_1$-saturated. 
\end{theorem}

\begin{proof}
We will write $I = I_\xc$ and $\mk = \mk_\xc$ for the duration of the proof.\footnote{Informal proof summary: suppose  
$p = \{ \vp_n(x,\bar{c}_n) : n < \omega \}$ 
witnesses $\xc$-unsuperstability. At stage $n$, copy the $n$th instance of shearing into the model
$\GEM(J, \Phi_{n-1})$, over what came before. 
The `copy' in $\GEM(J, \Phi_n)$ of the $\mk$-indiscernible sequence witnessing this instance of shearing 
has the same type as the `original', over what came before. 
Take a partial automorphism of the monster model fixing the parameters so far, sending the original sequence onto the 
copy. Replace $p$ by its image under this map; note the first $n-1$ formulas don't change. Keep going.} 
Let $M = \GEM(I, \Phi)$, and let $J \supseteq I$ be some $\aleph_0$-saturated member of $\mk$. 

Let $\langle I_n : n < \omega \rangle$, $\langle \bar{s}_n : n < \omega \rangle$, $\langle B_n : n < \omega \rangle$, 
and a partial type $p$ over $B$ be given from Definition \ref{m5a} to witness the $\xc$-unsuperstability of $T$.  
To review and fix notation, this means we start with the data of: 
\begin{itemize}
\item[(i)] $I_0 \subseteq \cdots \subseteq I_n \cdots$ are an increasing sequence of finite subsets of $I$, with $\bigcup_n I_n = I$. 
Without loss of generality (see \ref{c:monot}) $I_0 =  \emptyset$.  
\\ For each $n$, $\bar{s}_n$ is an enumeration of $I_n$, and let us assume $\bar{s}_n \tlf \bar{s}_{n+1}$. 
\item[(ii)] $B_0 \subseteq \cdots \subseteq B_n \cdots$ are an increasing sequence of 
nonempty, possibly infinite, sets of parameters in the 
monster model for $T$, with $B := \bigcup_n B_n$. 
For each $n$, there is a formula $\vp_{n}(\bar{x}, \bar{c}_{n})$ witnessing 
that $p \rstr B_{n}$ is a type which $(I_{n-1}, I_{n})$-shears over $B_{n-1}$.
\end{itemize}
For simplicity in the present proof, without loss of generality: 
\begin{itemize}
\item[(iii)]  Since $p$ is allowed to be a partial type, 
we may assume that:
\[ p = \{ \vp_{n}(\bar{x}, \bar{c}_{n}) : 1 \leq n < \omega \}. \]
\item[(iv)] Moreover, since a sequence which is $\mk$-indiscernible over some $C$ remains $\mk$-indiscernible over 
$C^\prime \subseteq C$, 
we may assume\footnote{recall notation: if $\lgn(\bar{c}) = k$, then 
$\{ \bar{c}_\ell : \ell < k \}$ is simply the set of its elements.}:  that $B_0 = \{ \bar{c}_{0,\ell} : \ell < \lgn(\bar{c}_0) \}$
and for each $n \geq 1$, $B_{n} = B_{n-1} \cup \{ \bar{c}_{n,\ell} : \ell < \lgn(\bar{c}_{n}) \}$. 
\item[(v)] Finally, for each $n$, to fix notation, 
say the shearing of $\vp_n$ is witnessed by 
$\mb^{n} = \langle \bar{b}^{n}_{\bar{s}} : \bar{s} \in {^{\omega>}(J[I_{n-1}])} \rangle$ in the 
monster model, which is $\mk$-indiscernible over $B_{n-1}$.  Note in particular that $\bar{b}^n_{\bar{s}_n} = \bar{c}_n$.  
(We could also have asked that $\bar{c}_{n-1} \tlf \bar{c}_n$.)
\end{itemize}
As our notation suggests, without loss of generality the $\aleph_0$-saturated sequence $J \supseteq I$ indexing each $\mb^n$ 
is the same $J$ mentioned in the second line of the proof.  As before, we may consider $M \rstr \tau(T)$ and $B$ as subsets of 
the monster model, and a priori they may have nothing to do with each other.\footnote{Most of the time we will work in $\mathfrak{C}_T$, though when we are expanding before applying the Ramsey property, it is really $\mathfrak{C}_{T_\Phi}$. 
Perhaps it is best to consider that $\mathfrak{C}_T$ denotes ``$\mathfrak{C}_{T_\Phi} \rstr \tau(T)$.'' 
Note also that there there is no connection asserted between 
the various $J$-indexed sequences witnessing dividing, or between these and possible skeletons of $\GEM$ models.
Moreover, the types of the various $J$-indexed sequences $\mb^n$ could certainly be different.}  

By induction on $n < \omega$ we will define $\Phi_n$ and $q_n$ (and auxiliary objects $\xr_n$, $\bar{F}_{\xr_n}$, $G_n$), which 
will be objects of the following kind. 

\br
\noindent \emph{for each} $n \geq 0$, 
\begin{enumerate}[(a)]
\item $\Phi_n \in \Upsilon_\xc[T]$, and $m < n$ implies $\Phi_m \leq \Phi_n$, and $\Phi_0 = \Phi$. 
\item $|\tau(\Phi_n)| \leq |\tau(\Phi_0)| + |T|$
\end{enumerate}
\emph{for each} $n \geq 1$, 
\begin{enumerate}[(b)]
\item[(b)] $q_n$ is a partial type of $\GEM_{\tau(T)}(I, \Phi_n)$,  which $(I_{n-1}, I_n)$-shears over the 
submodel $\GEM_{\tau(T)}(I_{n-1},\Phi_{n-1})$, witnessed by the sequence 
\[ \langle \bar{F}_{\xr_n}(\bar{a}_{\bar{s}}) : 
\bar{s} \mbox{ is from $J$ and realizes } \tpqf(\bar{s}_n, \bar{s}_{n-1}, I) ~ \rangle \]
where $\bar{F}_{\xr_n}$ is a 
sequence of $\ell(\bar{b}^n_{\bar{s}_n})$-many function symbols of $\tau(\Phi_n)$, each of arity 
$\ell(\bar{a}_{\bar{s}_n})$, where here 
$\langle a_j : j \in J \rangle$ denotes the skeleton. 
\item[(c)] each $q_n$ is contained in the image of $p$ under some partial automorphism of $\mathfrak{C}$, and 
$m < n$ implies $q_m \subseteq q_n$.  In slight abuse of notation, 
\[ q_n = G_{n}(G_{n-1}( \cdots G_1( ~~\{ \vp_m(\bar{x}, \bar{c}_m) : m \leq n \}  ~~) \cdots)) \] 
where the $G$'s are from item (f). 
\end{enumerate}
\noindent \emph{for each $n \geq 1$, auxiliary objects used in the construction}:
\begin{enumerate}
\item[(d)] $\xr_n \in \tsqf(J)$. 
\item[(e)] $\bar{F}_{\xr_n}$ are new function symbols of $\tau(\Phi_n)$ with domain $Q^J_{\xr_n}$, see (b). 
\item[(f)] $G_n$ is a partial automorphism of $\mathfrak{C}_T$, see (c).  
The domain of $G_1$ includes each of the shearing sequences  $\mb^i$ for $i<\omega$, \emph{thus} it also 
includes the domain $B$ of $p$. 
The domain of $G_n$ includes the range of $G_{n-1}$.  Finally, for $m<n$, $G_n$ 
leaves $\bar{c}_m$ fixed, in slight abuse of notation $G_n \rstr \bar{c}_m = G_m \rstr \bar{c}_m$. 
\end{enumerate}
Note that it will follow from (b) that $q_n$ shears ``up to stage $n$'' in $\GEM(J, \Phi_n)$, it will follow from 
(c) that $q_n$ extends to a partial type which continues to witness $\xc$-unsuperstability, and as we shall 
verify, it will follow from both that $\bigcup_n q_n$ is a partial type of $\GEM(J, \bigcup_n \Phi_n)$, or indeed of 
$\GEM(I, \bigcup_n \Phi_n)$, which 
witnesses $\xc$-unsuperstability, in a precise sense verified at the end of the proof.  When the skeleton is 
indexed by $I$ instead of $J$, the shearing sequences may not be contained in the model, but the parameters 
for the type are, which will be sufficient.

\br

\noindent\textbf{Stage $n=0$.}  Let $\Phi_0 = \Phi$. 

\br
\noindent\textbf{Stage $n=1$.}
Let $N_1 = \GEM(J, \Phi_0)$ with skeleton $\ma$. 
The formula 
\begin{equation} \label{e72} 
\vp_1(\bar{x}, \bar{c}_1)
\end{equation} 
$(\emptyset, \bar{s}_1)$-shears over $B_0$. The shearing is witnessed by $\mb^1$, and in particular, 
$\bar{b}^1_{\bar{s}_1} = \bar{c}_1$. 
Let $\xr_1 = \tpqf(\bar{s}_1, \emptyset, I)$. 
Let $\bar{F}_{\xr_1}$ be a sequence of $\ell(\bar{b}^1_{\bar{s}_1})$-many new $\ell(\bar{a}_{\bar{s}_1})$-ary function symbols. 
Considering $N_1$ as an elementary submodel of $\mathfrak{C}_{T_\Phi}$, interpret $\bar{F}_{\xr_1}$ so that 
\[ \bar{F}_{\xr_1}(\bar{a}_{\bar{s}}) = \bar{b}^1_{\bar{s}} ~ \mbox{ for each $\bar{s} \in Q^J_{\xr_1}$.} \]
(So in particular $\bar{F}_{\xr_1}(\bar{a}_{\bar{s}_1}) = \bar{b}^1_{\bar{s}_1} = \bar{c}_1$.)
Let $N_{1,1}$ be an elementary extension of $N_1$ in this larger language, which is closed under the functions 
$\bar{F}_{\xr_1}$  
and which is then also expanded to have Skolem functions. 

Apply the Ramsey property to $N_{1,1}$, $\ma$, $\Phi_0$.  
We obtain $\Phi_1 \geq \Phi_0$, and the functions from $\bar{F}_{\xr_1}$ are in $\tau(\Phi_1)$.  
Because of the Skolem functions, $\Phi_1$ is nice. 
In $\GEM(J, \Phi_1)$, the sequence 
\begin{equation}
\label{f:seq}
 \langle F_{\xr_1}(\bar{a}_{\bar{s}}) : \bar{s} \in Q^J_{\xr_1} \rangle 
\end{equation}
need no longer be identical to $\mb^1$, but because of the reflection clause in the Ramsey property 
and the fact that $\mb^1$ is $\mk$-indiscernible, this new sequence 
is also $\mk$-indiscernible and will have the same type as $\mb^1$. Let $G_1$ be a partial automorphism of $\mathfrak{C}_{T}$ 
whose domain includes $B$, which sends 
\[ \langle \bar{b}^1_{\bar{s}} : \bar{s} \in Q^J_{\xr_1} \rangle \] 
to the sequence from (\ref{f:seq}) in the natural way, i.e. 
\[ \bar{b}^1_{\bar{s}} \mapsto 
\bar{F}^{\GEM(J, \Phi_1)}_{\xr_1}(\bar{a}_{\bar{s}}) \mbox{ \hspace{5mm} for $\bar{s} \in Q^J_{\xr_1}$.} \]
In slight abuse of notation, write $G_1(p)$ for the image of the partial 
type $p$ after applying $G_1$. Then  
$G_1(p)$ is a type of $\ts(G_1(B))$ such that for
every $n > 0$,  we have that 
$G_1(p) \rstr G_1(B_n)$ is a type which $(I_{n-1}, I_n)$-shears 
over $G_1(B_{n-1})$, and moreover for $n=1$ this shearing is witnessed by the sequence (\ref{f:seq}) and the formula 
\[ \vp_1(\bar{x}, G_1(\bar{b}^1_{\bar{s}_1})) ~ \equiv ~ \vp_1(\bar{x}, \bar{F}_{\xr_1}(\bar{a}_{\bar{s}_1})) 
~\equiv~ \vp_1(\bar{x}, G_1(\bar{c}_1)). \] 
We define
\[ q_1 = \{ \vp_1(\bar{x}, G_1(\bar{c}_1)) \} \]
[i.e. the $G_1$-image of (\ref{e72})] 
so as $\bar{s}_1$ is a sequence of elements of $I$, by the equivalence just given, 
$q_1$ is a partial type of $\GEM_{\tau(T)}(I, \Phi_1)$ which $(\emptyset, I_1)$-shears 
over $\emptyset$, and is contained in an automorphic image of $p$.  This completes the base stage. 

\noindent\textbf{Stage $n=k+1$.}  
As the stage begins, we have a template $\Phi_{n-1}$,  partial automorphisms $G_1, \dots, G_{n-1}$, and a partial type\footnote{i.e. 
in the slight abuse of notation from above, $q_k$ is $G_1( \{ \vp_1 \} ) \cup  G_2(G_1( \{ \vp_2 \})) \cup \cdots \cup G_{n-1}(\cdots G_1( \{ \vp_{n-1} \} ) ~) $.  Or recalling item $(f)$ of the list at the beginning of the proof, $q_k$ is just $G_{n-1}(G_{n-2}(\cdots ( \{ \vp_1, \dots, \vp_{n-1} \} ) \cdots))$. } 
$q_{n-1}$ of $\GEM_{\tau(T)}(I, \Phi_{n-1})$,  
such that 
\begin{align*} q_{n-1} = &  \{~\vp_1(\bar{x}, G_1(\bar{c}_1)) ~\} \cup \{ ~\vp_2(\bar{x}, G_2(G_1(( \bar{c}_2)) ) ~\} \cup \cdots \\ 
&  \hspace{20mm} \cdots 
\cup \{ ~ \vp_{n-1}(\bar{x}, G_{n-1}(G_{n-2}(\cdots G_1 (\bar{c}_{n-1} ))) ~) ~ \}.
\end{align*}
By inductive hypothesis, 
\begin{itemize}
\item for each $1 \leq j < n$,  $q_j$ is a partial type with 
parameters from the submodel $\GEM_{\tau(T)}(I_j, \Phi_j)$; 
\item  $q_j$ is a partial type which $(I_{j-1}, I_{j})$-shears over $\GEM_{\tau(T)}(I_{j-1}, \Phi_{j-1})$. 
\item if we consider $q_j$ as a partial type in $\GEM_{\tau(T)}(J, \Phi_{j})$, 
there are a sequence of function symbols $\bar{F}_{\xr_j}$ of $\tau(\Phi_{j})$ 
such that  in this model, 
\[   \langle \bar{F}_{\xr_j}(\bar{a}_{\bar{s}} )   ~: ~~ {\bar{s}_{j-1}} ~^\smallfrown \bar{s}_{} \in Q^J_{\xr_j} \rangle \]
witnesses this shearing. 
\end{itemize}

To simplify notation, locally in this stage, write $G$ to abbreviate the composition $G_{n-1}\circ G_{n-2} \circ \cdots \circ G_1$. 
Let $M_n = \GEM(I, \Phi_{n-1})$.  
Let $N_n = \GEM(J, \Phi_{n-1})$. As $G$ is a partial automorphism, 
$\vp_n(\bar{x}, G(\bar{c}_n))$ is a formula which is consistent with $q_k$ and which 
$(I_{n-1}, I_n)$-shears over $G(B_{n-1})$,\footnote{Note that the set $G(B_{n-1})$ includes the domain of $q_k$. 
However, the set $G(B_{n-1})$ certainly need not include the algebraic closure of the domain of $q_k$, such as $\GEM(I_k, \Phi_k)$. 
Indiscernibility over this possibly larger or possibly different set will be guaranteed only after we let the Ramsey property 
make a better choice of $\mk$-indiscernible sequence for us.} as witnessed by $G(\mb^n)$.  Let 
\[  \xr_{n} = \tpqf(\bar{s}_{n}, \emptyset, I). \]
Let $\bar{F}_{\xr_n}$ be a sequence of $\ell(\bar{b}^n_{\bar{s}_n})$-many new $\ell(\bar{a}_{\bar{s}_n})$-ary function symbol(s).
Considering $N_n$ as an elementary submodel of $\mathfrak{C}_{T_\Phi}$, and recalling that 
$\bar{s}_{n-1}$ is the distinguished enumeration of $I_{n-1}$ (and an initial segment of $\bar{s}_n$), interpret them as follows. 
First, we consider elements coming from the skeleton of the form $\bar{a}_{\bar{s}}$ where $\bar{s} \in Q^J_\xr$ and 
$\bar{s}_{n-1} \tlf \bar{s}$. In this case, interpret so that 
\[ \bar{F}_{\xr_n} ( \bar{a}_{\bar{s}} ) = G(\bar{b}^n_{\bar{s} } ) 
~ \mbox{         for each $\bar{s} \in Q^J_{\xr_n}$ such that $\bar{s}_{n-1} \tlf \bar{s}$.            } \]
Next, we consider elements coming from the skeleton of the form $\bar{a}_{\bar{t}}$ where $\bar{t} \in Q^J_{\xr_n}$ 
and $\bar{t} \rstr \ell(\bar{s}_{n-1}) = \bar{s}^\prime$ for some $\bar{s}^\prime \in Q^J_{\xr_{n-1}}$ possibly different from $\bar{s}_{n-1}$. 
For each such $\bar{s}^\prime$, fix some partial automorphism of the monster model $H_{\bar{s}_{n-1} \mapsto \bar{s}^\prime}$ whose 
domain includes $G(B) \cup N_n \cup G(\mb^n)$ and which extends the automorphism of $N_n$ induced by sending $\bar{s}_{n-1}$ to 
$\bar{s}^\prime$ in the index model. Then, for this fixed $\bar{s}^\prime$, interpret $F_{\xr_n}$ so that 
\[ \bar{F}_{\xr_n}(\bar{a}_{{\bar{s}}}) = H_{\bar{s}_{n-1} \mapsto \bar{s}^\prime} (G(\bar{b}^n_{{\bar{s}}} ))  
~ \mbox{         for each $\bar{s} \in Q^J_{\xr_n}$ with $\bar{s}^\prime \tlf \bar{s}$.            } \]
Note that the reason to do the parallel expansion for all $\bar{s}^\prime \in Q^J_{\xr_n}$ is so that the Ramsey property will record the 
type of the $\mk$-indiscernible sequence correctly, over each $\bar{a}_{\bar{s}^\prime}$. 
Let $N_{n,1}$ be an elementary extension of $N_n$ in this larger language, which is closed under $F_{\xr_n}$ 
and which is then also expanded to have Skolem functions. 
Apply the Ramsey property to $N_{n,1}$, $\ma$, $\Phi_{n-1}$.  
We obtain $\Phi_n \geq \Phi_{n-1}$. Note that $F_{\xr_n} \in \tau(\Phi_n)$.  
Again because of the Skolem functions, we are assured $\Phi_n$ is nice. 

Just as in the base case, in the model $\GEM(J, \Phi_n)$, the sequence 
\begin{equation}
\label{fn:seq}
 \langle F_{\xr_n}(\bar{a}_{   \bar{s}       } ) : \bar{s} \in Q^J_{\xr_n}, ~\bar{s}_{n-1} \tlf \bar{s} \rangle 
\end{equation}
need no longer be identical to the sequence 
\begin{equation} 
\label{nx:seq} 
\langle G(\bar{b}^n_{ \bar{s} }) ~ :~ \tpqf(  \bar{s}, \bar{s}_{n-1}, J) = 
\tpqf(\bar{s}_{n}, \bar{s}_{n-1}, I) \rangle 
\end{equation} 
but because of the reflection clause in the Ramsey property 
and the fact that $(\ref{nx:seq})$ is $\mk$-indiscernible, this new sequence $(\ref{fn:seq})$ 
is also $\mk$-indiscernible and will have the same type as $(\ref{nx:seq})$. 
[We really use the $J[I_{n-1}]$ in the definition of $\mb_n$ from Definition \ref{e5} here: 
the $\mk$-indiscernible sequence we use is indiscernible over $\bar{a}_{\bar{s}_{n-1}}$, i.e. ``over $I_{n-1}$''.]
Let $G_n$ be a partial automorphism of $\mathfrak{C}_{T}$ 
which is the identity on the domain of $q_{n-1}$, 
whose domain includes $B$, and which sends $(\ref{nx:seq})$ to $(\ref{fn:seq})$ in the natural way: 
\[ G( \bar{b}^n_{\bar{s}} )  \mapsto F_{\xr_n}(\bar{a}_{\bar{s}}).\] 
Then $q_n = G( p_{n-1} \cup \{ \vp_n \} )$ is a partial type of $\GEM(I, \Phi_n)$, indeed of $\GEM(I_n, \Phi_n)$, which 
satisfies the inductive hypothesis.  

This completes the inductive step, and so the induction. 

\br
\noindent \textbf{Verification}.   Let $\Psi \geq \Psi_* = \bigcup_n \Phi_n$ and let $q = \bigcup_n q_n$. 
Let $M = \GEM(I, \Psi)$ and let $N = \GEM(J, \Psi)$. 
Then $q$ is a partial type of $M \rstr \tau(T)$.  Let us show it is not realized in $M$. 
Assume for a contradiction that it were realized, say by $\bar{d}$.  Then for some $k<\omega$, 
$\bar{d} \subseteq \GEM(I_k, \Phi_k)$. 
We know that in $\GEM(J, \Phi_{k+1})$ there is a formula of $q$ which $(I_{k}, I_{k+1})$-shears 
over $\GEM(I_k, \Phi_k)$. Since the sequence $\bar{d}$ cannot realize the type in this larger model $N$, a fortiori 
it cannot realize the type in the smaller model $M$. 
\end{proof}

\begin{rmk}
The proof of Theorem $\ref{t:dir2}$ builds 
a type $p$ which does not have a weak $\bar{t}_*$-definition for any 
finite $\bar{t}_*$ in $I$, and moreover cannot have one in any $\GEM(I, \Psi)$ for $\Psi \geq \Phi_*$.  The failure of 
$\aleph_1$-saturation in $\GEM_{\tau(T)}I, \Psi)$ for any $\Psi \geq \Phi_*$ will always be due to this $p$ $($of course 
other types may be omitted as well$)$.
\end{rmk}

\begin{cor}[Local unsuperstability] \label{t:dir2a} 
Let $\Delta$ be a set of formulas of $T$. 
Suppose $\xc$ is a countable context, and assume $(T, \Delta)$ is $\xc$-unsuperstable.  
For every $\Phi \in \Upsilon_\xc[T]$ there is $\Phi_* \geq \Phi$ with $|\tau(\Phi_*)| \leq |T| + |\tau_\Phi| + \aleph_0$ such that 
for every $\Psi \geq \Phi_*$ we have  
$\GEM_{\tau(T)}(I, \Psi)$ is not $\aleph_1$-saturated, in particular, it will omit a $\Delta$-type over a countable set.  
\end{cor}

\begin{disc}
\emph{In \ref{t:dir2} and \ref{t:dir2a} we make no assumptions on the size of the language.  We can require $|\tau(\Phi_*)| \leq |\tau(\Phi)| + |T|$, with 
no requirement on 
$\tau(\Psi)$.}
\end{disc}

\begin{claim} 
\label{needed3}
There exists a countable context $\xc$ such that if $T$ is any theory which is supersimple in the usual sense, 
then $T$ is $\xc$-superstable, and if $T^\prime$ is any theory which is not supersimple in the usual sense, 
then $T$ is $\xc$-unsuperstable. 
\end{claim}

\begin{proof}
The existence of such a $\xc$ 
is given by Theorem \ref{n5} below. 
\end{proof}

\vspace{2mm}

\subsection*{Discussion and examples}
Several examples may shed light on the interaction between ambient shearing and its appearance in $\GEM$-models. 
We continue under Hypothesis \ref{m2x}. 

From the proofs of this section, we see: 

\begin{concl} \label{m-types}
In determining whether a theory $T$ is $\xc$-superstable, for a countable context $\xc$, it suffices to consider $m$-types for 
some fixed $m<\omega$, e.g. $m=1$.
\end{concl}

\begin{proof}
We have established an equivalence in terms of saturation, and for saturation this is true. 
\end{proof}

\begin{expl}
It may be the case that for every nonalgebraic formula $\vp(x,c)$ of some $\GEM_{\tau(T)}(I, \Phi)$, if we write 
$\vp$ as $\vp(x,\bar{\sigma}(\bar{a}_{\bar{t}})) $
for some sequence $\bar{\sigma}$ of $\tau(\Phi)$-terms and $\bar{a}_{\bar{t}}$ of the appropriate length, then 
\[ \{  \vp(x,\bar{\sigma}(  \bar{a}_{\bar{s}} )  )  : \tpqf(\bar{s}, \emptyset, J) = \tpqf(\bar{t}, \emptyset, I) \} \]
is consistent in $\GEM(J, \Phi)$ for every $J \supseteq I$, \emph{even though} some of these formulas divide, 
thus shear. 
\end{expl}

\begin{proof}
Let $T$ be the theory of an equivalence relation with infinitely many infinite classes. By quantifier elimination it 
suffices to consider $\vp(x,c) = E(x,c)$. Suppose we have set up the 
$\GEM$-model $M$ so that the skeleton $\ma$ is an infinite sequence within a single class, with Skolem functions $\langle F_n : n < \omega \rangle$
interpreted so that $F^M_n$ copies $\ma$ over to the $n$-th class. Then $\vp$ clearly divides (so a fortiori shears), and even does so along an 
indiscernible sequence in $\GEM_{\tau(T)}(I, \Phi)$. Still, the set of formulas in the statement will remain consistent since, in $M$, no two 
$s,t$ which share a quantifier-free type satisfy $\neg E(a_s, a_t)$.  However, the proof of Theorem \ref{t:dir2} shows that we may find 
$\Psi \geq \Phi$ in which an analogous instance of dividing does occur ``along the skeleton''.  Informally, first choose an elementary extension of $M$ 
in which there are many nonstandard classes, interpret a new function symbol $G$ which maps each $a_t$ from $\ma$ to an element in a distinct nonstandard $E$-class, 
and apply the Ramsey property to obtain $\Psi$. Then letting $t$ be any nonalgebraic element of  any $\aleph_0$-saturated 
$J \supseteq I$ and $\vp(x,c)$ be any formula of $\GEM_{\tau(T)}(I, \Phi)$ such that 
$c = G(a_t)$ for some $t \in I$, the sequence $\langle G(a_s) : \tpqf(s,\emptyset, J) = \tpqf(t, \emptyset, I)$ will be in distinct $E$-equivalence classes, 
and so will witness the dividing of $\vp(x,c) = E(x,c)$. 
\end{proof}

\br
This example may be easily modified with finitely many equivalence relations to produce examples where the given $\GEM$-model 
does or does not ``witness'' the natural ``superstability rank'' (for dividing or shearing), and even more, showing the 
importance of the template $\Psi$ in ``witnessing'' shearing:

\begin{expl}
Let $\xc$ be a countable context, so $I = I_\xc$. 
It may be the case that for every nonalgebraic formula $\vp(x,c)$ of some $\GEM_{\tau(T)}(I, \Phi)$, if we write 
$\vp$ as $\vp(x,\bar{\sigma}(\bar{a}_{\bar{t}})) $
for some sequence $\bar{\sigma}$ of $\tau(\Phi)$-terms and $\bar{a}_{\bar{t}}$ of the appropriate length, then 
\[ \{  \vp(x,\bar{\sigma}(  \bar{a}_{\bar{s}} )  )  : \tpqf(\bar{s}, \emptyset, J) = \tpqf(\bar{t}, \emptyset, I) \} \]
is consistent in $\GEM(J, \Phi)$ for every $J \supseteq I$, \emph{even though} $T$ is not superstable for the context $\xc$.  
\end{expl}

\begin{proof}
Let $T$ be the theory of infinitely many equivalence relations, $\{ E_i : i < \omega \}$, where each $E_i$ has 
infinitely many infinite classes and for each $i<\omega$, each $E_{i+1}$-class is the union of infinitely many $E_i$-classes. 
This theory is stable but not superstable. 
The previous example extends naturally to this case, 
provided we have set up the $\GEM$-model $M$ so that the skeleton $\ma$ is an infinite sequence within a single class for each $E_i$, with Skolem functions $\langle F_{n,i} : n < \omega \rangle$ interpreted so that $F^M_{n,i}$ copies $\ma$ over to the $n$-th class of the $i$-th equivalence relation. 
\end{proof}

\vspace{5mm}

\section{The separation theorem}
\setcounter{theoremcounter}{0}

\begin{theorem}[Separation Theorem] \label{firstmainth}
Let $T_0, T_1$ be any two theories, without loss of generality in disjoint signatures, and of any size. Suppose there exists a countable context $\xc$ such that $T_0$ is $\xc$-superstable and 
$T_1$ is $\xc$-unsuperstable.  Then for any $\Phi \in \Upsilon_\xc$, for arbitrarily large $\mu$, there exists 
$\Psi \geq \Phi$ with $\Psi \in \Upsilon_\xc[T_0] \cap \Upsilon_\xc[T_1]$ such that 
writing $M = \GEM(I, \Psi)$, we have that $M \rstr \tau(T_0)$ is $\mu$-saturated but $M \rstr \tau(T_1)$ is not $\aleph_1$-saturated. 
\end{theorem}

\begin{proof} Let $\kappa$ be any infinite cardinal. 
Let $\Phi \in \Upsilon_\xc$ be given. By applying Observation \ref{cont-for-t} twice, if necessary, we may assume 
$\Phi \in \Upsilon_\xc[T_0] \cap \Upsilon_\xc[T_1]$.   
By Theorem \ref{t:dir2}, we may find 
$\Phi_1 \geq \Phi$ such that $\GEM_{\tau(T_1)}(I, \Phi^\prime)$ will not be $\aleph_1$-saturated for any $\Phi^\prime \geq \Phi_1$. 
Next, 
choose $\mu$ and $\lambda$ so that $\mu \geq \kappa$ and $\mu, \lambda$ satisfy the hypotheses of Claim \ref{m20a}. 
Apply Claim \ref{m20a} to find $\Phi_2 \geq \Phi_1$ so that $\GEM_{\tau(T_0)}(I, \Phi_2)$ is $\mu$-saturated. 
Then $\Psi = \Phi_2$ is as desired. 
\end{proof} 

\begin{cor} \label{st-cor}
Let $\xc = (I, \mk)$ be a countable context.  If $T_0$ is superstable for $\xc$ and $T_1$ is not superstable for $\xc$, 
for \emph{every} theory $T_*$ interpreting both of them $($without loss of generality in disjoint signatures$)$, 
and for arbitrarily large $\mu$, there is 
a model $M_* \models T_*$ such that $M_* \rstr \tau(T_0)$ is $\mu$-saturated but 
$M_* \rstr \tau(T_1)$ is not $\aleph_1$-saturated. 
\end{cor}

Note that one genre of corollary of the Separation Theorem is to point out various constraints on models arising as $\GEM$ models. 

\begin{concl}
If $\xc$ is a countable context and $T$ is $\xc$-unsuperstable, then for some $\Phi$, for every $\Psi \geq \Phi$, no $\GEM$ model $\GEM(I_\xc, \Psi) \models T$ is $\aleph_1$-saturated, already when restricted to $\tau(T)$. 
\end{concl}

\vspace{5mm}

\section{Consequences for $\tlf^*$} \label{s:tlf}
\setcounter{theoremcounter}{0}

In this section we apply \ref{st-cor} to obtain a series of results about the structure of the interpretability order $\tlf^*_1$.   
The setup suggests a method for obtaining various further results. 
We emphasize that 
all results are in ZFC. 

\begin{conv} 
All theories in this section are complete. 
\end{conv}

Recall that $\tlf^*_1$ means (for readers used to all three subscripts): $\tlf^*_{\lambda, \chi, \kappa}$, i.e. 
for all sufficiently large $\lambda$, for $\chi = |T_0|+|T_1|$, for $\kappa = 1$ (so ``for every $1$-saturated model'' 
abbreviates ``for every model'').  For a complete definition and motivation, 
see sections 1-2 of \cite{MiSh:1124}.

First we recall a fact which spells out the sense in which $\tlf^*_1$ naturally refines Keisler's order $\tlf$. In the context of Keisler's order, writing $\tlf_\lambda$ means that we restrict to regular ultrafilters on $\lambda$.

\begin{fact}[\cite{MiSh:1124} Corollary 2.11] \label{f10} If for arbitrarily large $\lambda$ we have $T_0 \tlf_\lambda T_1$ in Keisler's order, then 
$\neg (T_1 \tlf^*_1 T_0)$. 
\end{fact}

\begin{claim} \label{c5} 
Let $T_0$, $T_1$ be complete countable theories and let $\xc$ be a countable context. 
Suppose $T_0$ is $\xc$-superstable and $T_1$ is $\xc$-unsuperstable for $\xc$. Then 
\[ \neg (T_1 \tlf^*_1 T_0 ). \]
\end{claim}

\begin{proof} 
This just applies Corollary \ref{st-cor} to the definition of $\tlf^*_1$.  
\end{proof}

\begin{cor}
Let $T_0$, $T_1$ be complete countable theories.. 
Suppose $T_0$ is supersimple and $T_1$ is not supersimple.  Then 
\[ \neg (T_1 \tlf^*_1 T_0 ). \]
\begin{proof}
Claim \ref{c5} and Claim \ref{needed3}.  
\end{proof}

\end{cor}
For the next few results, Keisler's order is invoked in the proofs so we restrict to countable theories (for which Keisler's order is defined).

\begin{lemma}
Let $T_i$ be superstable with the fcp and let $T_j$ be strictly stable nfcp. Then $T_i$, $T_j$ are $\tlf^*_1$-incomparable. 
\end{lemma}

\begin{proof}
$T_j \tlf_\lambda T_i$ for arbitrarily large $\lambda$ in Keisler's order, so by \ref{f10} $\neg (T_i \tlf^*_1 T_j)$. For the other direction, let 
$\xc$ be from \ref{needed3}. 
Then $T_i$ is $\xc$-superstable and $T_j$ is $\xc$-unsuperstable, so by \ref{c5}, 
$\neg (T_j \tlf^*_1 T_i)$. 
\end{proof}

\begin{concl}
The interpretability order $\tlf^*_1$ is not linear even on the stable theories. 
\end{concl}

\begin{theorem} \label{t-sim} Let $T_i$, $T_j$ be complete countable theories. Suppose $T_i$ is strictly stable.  
Suppose $T_j$  is supersimple unstable.   Then $T_i$ and $T_j$ are $\tlf^*_1$-incomparable. 
\end{theorem}

\begin{proof}
We know $T_i \triangleleft_\lambda T_j$ in Keisler's order for arbitrarily large $\lambda$, so $\neg (T_j \tlf^*_1 T_i)$. 
For the other direction, let $\xc$ be from $\ref{needed3}$. 
$T_i$ is not $\xc$-superstable, because it is strictly stable, but 
$T_j$ is $\xc$-superstable, because it is supersimple. So by \ref{c5}, 
$\neg (T_j \tlf^*_1 T_i)$. 
\end{proof}

Theorem \ref{t-sim} has various immediate, but more quotable, corollaries. 
Recall from \cite{MiSh:1124} that $\trg$ (theory of the random graph) is $\tlf^*_1$-minimum among unstable theories. 

\begin{concl} It is not the case that 
all stable theories are below all unstable theories in $\tlf^*_1$. 
\end{concl}

\begin{cor}
Let $T$ be countable and strictly stable. Then $T$ and $\trg$ are $\tlf^*_1$-incomparable. 
\end{cor}

\begin{cor}
Let $T$ be Hrushovski's strictly stable $\aleph_0$-categorical pseudoplane. Then $T$ and $\trg$ are $\tlf^*_1$-incomparable.
\end{cor}

\begin{concl}
The order $\tlf^*_1$ is not linear even on the countable $\aleph_0$-categorical theories.   Moreover, 
it is not linear even on countable $\aleph_0$-categorical graphs. 
\end{concl}

\begin{disc}
\emph{Of course, a priori we do not know that $\neg( T \tlf^*_1 T^\prime)$, or even just $\tlf$, means that there is a way to see the difference 
via superstability for some context. One could naturally define a new triangle ordering saying that $T_1$ below $T_2$ means that if $T_2$ is $\xc$-superstable for some countable context $\xc$ then so is $T_1$.}
\end{disc}

\vspace{5mm}

\section{Simple and supersimple} \label{s:simple}

\setcounter{theoremcounter}{0}

In this section we characterize the pairs $(T, \Delta)$ which are $\xc$-superstable for \emph{some} countable context $\xc$. 

\begin{disc}
\emph{In \cite{MiSh:1124} Lemma 7.10, we proved that   
that for any simple theory $T$ with $\kappa(T) \leq \kappa$, for arbitrarily large $\mu$, 
for a certain context (which took as a parameter $\kappa$), for every $\Phi$, there 
was $\Psi \geq \Phi$ such that $M = \GEM_{\tau(T)}(I, \Psi)$ is $\mu$-saturated.  Notice that this does not contradict the  
results of \S \ref{s:unss}, since that section used countability of the context in an essential way (informally, 
this is the case ``$\kappa(T) = \aleph_0$'').  Some further remarks in this line are given in \S \ref{other-kappa}.} 
\end{disc}

We will use the following index model class, which is Ramsey, \cite{MiSh:1124} Fact 3.20. 
Note that we demand $\mk_\xc$ have a linear order. 
\begin{defn} \label{n5def} \emph{ }
\begin{enumerate}
\item $\mk^{tr}_\kappa$ is the class of trees with $\kappa$ levels and lexicographic order which 
are normal, meaning 
that a member $\eta$ at a limit level is determined by 
\[ \{ \nu : \nu \tlf \eta \}. \]  
$($So the tree has the function $\cap (\eta, \nu) = \min \{ \rho : \rho \tlf \nu, \rho \tlf \eta \}$.$)$
\item We call $I \in \mk^{tr}_\kappa$ standard when the $i$th level, $P^I_i$, of $I$ consists of sequences of length $i$ and $n \in P_i$, $j<i$, 
$\eta \rstr j \in P_j$ and $\eta \rstr j \tlf_I \eta$, so every $I \in \mk^{tr}$ is isomorphic to a standard one 
$($this is justified by the assumption of normality$)$. 
\end{enumerate}
\end{defn}

\begin{fact}[\cite{MiSh:1124} 7.12, update] \label{n14}
Let $I \in \mk^{tr}_\kappa$ be standard with universe ${^{\omega > }{\{ 0 \}}}$. Suppose $\Delta$ is a set of formulas of $T$ 
such that every $\Delta$-type in every model of $T$ does not fork over some finite set. 
Then for every $\Phi \in \up$, there is $\Psi \in \up$ with 
$\Phi \leq \Psi$ such that $M = \GEM_{\tau(T)}(I, \Psi)$ is $\mu$-saturated for $\Delta$-types.  
\end{fact}

We will need the following local update of \cite{MiSh:1124} Lemma 7.12. 

\begin{defn} 
Recall that  $\kappa_{\operatorname{loc}}(T) = \aleph_0$ means that for every formula $\vp$, every 
$\vp$-type does not fork over a finite set, while $\kappa(\vp, T)  = \aleph_0$ means that 
every $\vp$-type does not fork over a finite set. 
\end{defn}

\begin{lemma} \label{t:simple-local}
Let $T$ be any complete theory and let $\vp$ be a formula of $T$ which is simple and $\kappa(\vp, T) = \aleph_0$. 
Then  
$(T, \{ \vp, \neg \vp \})$ is $\xc$-superstable for the countable context $({^{\omega >}\{0\}}, \mathcal{K}^{tr}_\kappa)$. 
\end{lemma}

\begin{proof}
Let $T$, $\vp$ be given and let $\Delta = \{ \vp, \neg \vp \}$.  
Let $\xc = (I, \mk)$ where $\mk = \mk^{tr}_\kappa$ and $I$ is standard with universe ${^{\omega >}\{0\}}$, so is  
a single branch. 
Suppose for a contradiction that $(T, \Delta)$ is not $\xc$-superstable.
Then Fact \ref{n14} gives us a $\Phi_*$ such that for every $\Psi \geq \Phi_*$, 
$\GEM_{\tau(T)}(I, \Psi)$ omits some $\vp$-type over a countable set.  For $\Psi \geq \Phi_*$ given by 
Fact \ref{n14}, we get a contradiction. 
\end{proof}

\begin{theorem} \label{n20a}
Let $T$ be first order complete and $\vp$ a formula of $T$. Then the following are equivalent:
\begin{enumerate}
\item[(a)] $\vp$ is simple and $\kappa(\vp, T) = \aleph_0$. 
\item[(b)] $\xc$-supersimple where $\mk_\xc = \mk^{tr}_\kappa$ and $I_\xc = {^{\omega>}\{0\}}$. 
\item[(c)] for some countable context $\xc$, $(T, \{ \vp, \neg \vp \})$ is $\xc$-superstable. 
\end{enumerate}
\end{theorem}

\begin{proof}[Proof of Theorem \ref{n20a}] 

(a) implies (b): 
apply Lemma \ref{t:simple-local}. 

(b) implies (c): immediate. 

(c) implies (a): we prove the contrapositive. If $\vp$ is not simple or $\kappa(\vp, T) > \aleph_0$, 
this amounts to saying that saying that $(T, \Delta)$ is not superstable for $\Delta = \{ \vp, \neg \vp \}$ so we may apply 
 \ref{needed2}. 
\end{proof}

\begin{cor}
Let $T$ be first order complete. Then the following are equivalent:
\begin{enumerate}
\item[(a)] $T$ is simple and $\kappa_{\operatorname{loc}}(T) = \aleph_0$. 
\item[(b)] for some countable context $\xc$, for every $\vp$, $(T, \{ \vp, \neg \vp \})$ is $\xc$-superstable. 
\end{enumerate}
\end{cor}

\br

\begin{theorem} \label{n5}  
Suppose that $T$ is a complete first order theory. Then the following are equivalent:
\begin{enumerate}
\item[(a)] $T$ is supersimple, i.e. every $p \in \ts(M)$ does not fork over a finite set for $M \in \operatorname{Mod}_T$. 
\item[(b)] $T$ is $\xc$-supersimple where $\mk_\xc = \mk^{tr}_\kappa$ and $I_\xc = {^{\omega>}\{0\}}$. 
\item[(c)] There is some countable context $\xc$ for which $T$ is $\xc$-supersimple. 
\end{enumerate} 
\end{theorem}

\begin{proof}

(a) implies (b): Assume $T$ is supersimple and let $\xc = ({^{\omega>}\{0\}}, \mk^{tr}_\kappa)$. 
Assume for a contradiction that $T$ were not $\xc$-superstable. By the Separation Theorem, there would be a 
template $\Phi$ such that for every $\Psi \geq \Psi$, $\GEM(I, \Psi)$ is not $\aleph_1$-saturated. 
This contradicts Fact \ref{n14}. 

(b) implies (c): immediate. 

(c) implies (a): for the contrapositive, assume $T$ is not supersimple. Then by Claim \ref{needed2}, $T$ is not $\xc$-superstable for any countable 
context $\xc$.  
\end{proof}

We end this section with a motivating example for the work above and for \cite{MiSh:F2061}. 
\emph{Reminder: all contexts are countable.}
Fix for awhile $\xc = (I, \mk) = (I_\xc, \mk_\xc)$ some countable context and 
we shall investigate how $\xc$-dividing may arise for $\trg$ inside a $\GEM$-model.\footnote{This is similar to 
\cite{MiSh:1124} Claim 5.10, though slightly more general.}
Consider $M = \GEM(I, \Phi)$, where $\Phi \in \Upsilon[\trg]$ thus $(M, R^M) \models \trg$, 
and let $p \in \ts(M \rstr \tau(\trg))$ be a nonalgebraic type. Fix $J$ such that $I \subseteq J \in \mk$ and $J$ is $\aleph_0$-saturated. 
Let $N  = \GEM(J, \Phi)$. By our assumption that all contexts are nice, see \ref{c:nice}, $M \preceq N$, so we will identify the sequence 
$\langle \bar{a}_t : t \in I \rangle$ which generates $M$ with a subsequence of $\langle \bar{a}_t : t \in J \rangle$. 

By quantifier elimination, $p$ is equivalent to  $\{ R(x,b_\alpha)^{\ii_\alpha} \land x \neq b_\alpha ~:~ \alpha < \kappa \}$ for some $\kappa$, where each 
$\ii_\alpha \in \{ 0, 1 \}$.  As $M$ is generated by $\{ \bar{a}_t : t \in I \}$, each $b_\alpha$ 
may be written as 
$\sigma^M_\alpha(\bar{a}_{\bar{t}_\alpha})$ for some $\tau(\Phi)$-term $\sigma_\alpha$ and some $\bar{t}_\alpha \in \inc(I)$. 
This representation may, of course, not be unique. Since we are ultimately looking for a criterion that will prevent $\xc$-dividing, there is 
a priori no harm in choosing our enumeration to include all such representations. That is, 
without loss of generality, for some $\kappa = \kappa + |\tau(\Phi)|$, 
\begin{equation} \label{eq:p}
p(x) \equiv \{ R(x, \sigma^M_\alpha(\bar{a}_{\bar{t}_\alpha}))^{\ii_\alpha} \land 
x \neq \sigma^M_\alpha(\bar{a}_{\bar{t}_\alpha}) : \alpha < \kappa \}. 
\end{equation}  where for each $\alpha < \kappa$ and $\bar{t} \in {^{\omega >}I}$ if $\sigma^M(\bar{a}_{\bar{t}}) = \sigma^M_\alpha(\bar{a}_{\bar{t}_\alpha})$, 
then 
\begin{equation} \label{eq:r}
\mbox{ for some $\beta < \kappa$, $\sigma_\beta = \sigma$ and $\bar{t}_\beta = \bar{t}$. }
\end{equation}
Why? {This simply says we may choose to include all ways of generating the mentioned elements from the skeleton.}  
Note this may increase the length of the enumeration, but will not change the size of the type in $\tau(\trg)$.  
[From the point of view of $M \rstr \tau(\trg)$, it may appear that we have 
repeated various formulas of the form 
$R(x,b_\alpha)^{\ii_\alpha} \land x \neq b_\alpha$ many times, because we have listed an instance for each way of 
writing $b_\alpha$ in $M$ in terms of the skeleton; whereas from the point of view of 
our enumeration which has access to $\tau(\Phi)$, for each $b_\alpha$ there are potentially $|I| + |\tau(\Phi)|$ such 
representations.]

Recalling our fixed $\aleph_0$-saturated $J$ extending $I$ and its associated 
$N = \GEM(J, \Phi)$, we ask about potential $\xc$-dividing. 
Working in $N$, consider the set of formulas 
\begin{align*} \label{eq:q}
q(x) = q_{I_0}(x) =  & \{ R(x, \sigma^N_\alpha(\bar{a}_{\bar{u}}))^{\ii_\alpha} 
\land x \neq \sigma^M_\alpha(\bar{a}_{\bar{u}}) : 
 \alpha < \kappa, \\ & \mbox{ \hspace{5mm} } \tpqf(\bar{u}, I_0, J) = 
\tpqf(\bar{t}_\alpha, I_0, I) \}. 
\end{align*} 
To show $q(x)$ is consistent, it would suffice to check that whenever 
\begin{equation} 
R(x, \sigma^N_\alpha(\bar{a}_{\bar{v}}))^{\ii_\alpha}  \in q \mbox{ and } R(x, \sigma^N_\beta(\bar{a}_{\bar{w}}))^{\ii_\beta}  \in q 
\end{equation}
we have that {if} $\sigma^N_\alpha(\bar{a}_{\bar{v}}) = \sigma^N_\beta(\bar{a}_{\bar{w}})$ {then} $\ii_\alpha = \ii_\beta$. 
Suppose this fails. That is, suppose for some suitable $\alpha, \beta, \bar{v}_*, \bar{w}_*$ which we fix for awhile, 
$q$ contains the contradictory formulas 
\begin{equation}
\label{eq:as1} R(x,\sigma^N_\alpha(\bar{a}_{\bar{v}_*}))  \mbox{ and } \neg R(x,\sigma^N_\beta(\bar{a}_{\bar{w}_*})). 
\end{equation}
In other words, 
\begin{equation} 
\label{eq:as}
\sigma^N_\alpha(\bar{a}_{\bar{v}_*}) = \sigma^N_\beta(\bar{a}_{\bar{w}_*}) = {b} 
\mbox{ but } \ii_\alpha \neq \ii_\beta
\mbox{ (here w.l.o.g. $\ii_\alpha = 1$, $\ii_\beta = 0$). }
\end{equation} 
We may visualize what has happened as follows. In $N$, there is a ``positive line'' consisting of the elements 
\[ \posl = \{ \sigma^N_\alpha(\bar{a}_{\bar{v}}) :  \tpqf(\bar{v}, I_0, J) = \tpqf(\bar{t}_\alpha, I_0, I) \} \subseteq \dom(N) \]
and a ``negative line'' consisting of the elements 
\[ \negl = \{ \sigma^N_\beta(\bar{a}_{\bar{w}}) :  \tpqf(\bar{w}, I_0, J) = \tpqf(\bar{t}_\beta, I_0, I) \} \subseteq \dom(N) \]
and we have that  
\begin{equation} \label{eq:pl}
\posl \cap \negl \neq \emptyset, \mbox{ witnessed by } b = \sigma^N_\alpha(\bar{a}_{\bar{v}_*}) = \sigma^N_\beta(\bar{a}_{\bar{w}_*}). 
\end{equation}
However, it is also important to notice that both ``lines'' have ``points from $I$'', 
and that these are not the point(s) of intersection:\footnote{If $p$ is a complete type in $M$, 
then we will have the stronger statement that ``the restrictions of $\posl$ and $\negl$ to $I$ have no intersection,'' i.e. 
$\{ \sigma^N_\alpha(\bar{a}_{\bar{v}}) :  \tpqf(\bar{v}, I_0, I) = \tpqf(\bar{t}_\alpha, I_0, I) \} \cap  
\{ \sigma^N_\beta(\bar{a}_{\bar{w}}) :  \tpqf(\bar{w}, I_0, I) = \tpqf(\bar{t}_\beta, I_0, I) \} = \emptyset$, but we do not need this here.}
\begin{equation} \label{eq:7} \sigma^N_\alpha(\bar{a}_{\bar{t}_\alpha}) \neq \sigma^N_\beta(\bar{a}_{\bar{t}_\beta}) 
\end{equation} 
else our original $p$ would be inconsistent. (So both $\posl$ and $\negl$ have size $\geq 2$.) 

To see the key property of $\xc$  
hidden in this picture, let us pull back the picture from $N$ to a picture on $I$. Definition \ref{d:circ} 
may be suggested by two observations. First, writing $\xr_\alpha = \tpqf(\bar{t}_\alpha, I_0, I)$, we have that ``$\sigma^N_\alpha(\bar{a}_{\bar{v}}) = \sigma^N_\alpha(\bar{a}_{\bar{u}})$'' is an equivalence relation on $Q^J_{\xr_\alpha}$ (asserting that $\bar{v}, \bar{u}$ are equivalent), 
and similarly for $\sigma^N_\beta$ and $\xr_\beta$. 
Second, the fact that $\bar{t}_\alpha$ and $\bar{t}_\beta$ may have different types is a red herring (as will be explained).
The third formula, $F$, will give the analogue of (\ref{eq:as}) ``points of intersection.''

\begin{defn} \label{d:circ}
The context $\xc = (I, \mk)$ has $\xc$ has property $\ocirc$ when:

\begin{quotation}
\noindent  For every finite $I_0 \subseteq I$ there is a finite $I_1$ with 
$I_0 \subseteq I_1 \subseteq I$, letting $\bar{s}$, $\bar{t}$ list 
$I_0$, $I_1$ respectively, such that for any $\aleph_0$-saturated $J \supseteq I$ there exist quantifier-free 
$($possibly infinitary$)$ formulas of $\tau(\mk)$ called $F(\bar{x}_1, \bar{x}_2; \bar{y})$,  
$E_1(\bar{x}_1, \bar{x}_2; \bar{y})$, 
$E_2(\bar{x}_1, \bar{x}_2; \bar{y})$, such that
 $\ell(\bar{x}_1) = \ell(\bar{x}_2) = \lgn(\bar{t})$, $\lgn(\bar{y}) = \lgn(\bar{s})$, and:  
\begin{itemize}
\item[(i)] for $\ii = 1, 2$ $E_\ii(\bar{x}_1, \bar{x}_2; \bar{s})$ defines an equivalence relation on 
$Y_{\bar{s}} = \{ \bar{t}^\prime \in {^{\lgn(\bar{t})}J} : 
\tpqf(\bar{t}^\prime, \bar{s}, J) = \tpqf(\bar{t}, \bar{s}, J) \}$.
\item[(ii)] $F(\bar{x}_1, \bar{x}_2; \bar{y})$  
defines a nonempty one to one partial function from $Y_{\bar{s}}/E_1(-,-; \bar{s})$ to $Y_{\bar{s}}/E_2(-,-; \bar{s})$, and 
\item[(iii)] $F$ has no fixed points, in other words for no $\bar{t} \in Y_{\bar{s}}$ is it the case that $F(\bar{t}, \bar{t}; \bar{s})$. 
\end{itemize}
\end{quotation}
\end{defn}

In a manuscript in preparation, we use property $\ocirc$ to characterize countable contexts for which the random graph is 
$\xc$-unsuperstable \cite{MiSh:F2061}. 

\br

\section{Further remarks and open questions}  \label{other-kappa}

\setcounter{theoremcounter}{0}

As noted above, the results in the present paper dealing with countable contexts may be seen as the case of ``$\kappa = \aleph_0$.'' 
In this section we record the natural extensions of the definitions and theorems to the 
case of arbitrary $\kappa$ without proofs, noting the main work of the proofs is already done in the case $\kappa = \aleph_0$ 
(so even though we plan to give some details in future work, 
it is worth stating these versions  here for the 
interested reader). 
Complementary to the above comments: 

\begin{disc}
Beginning with \S \ref{s:unss}, we use repeatedly that $I$ is a countable context. ``Countably generated'' is likely enough, but one would have to 
check carefully.
\end{disc}

\begin{defn} \label{p2}  \emph{ }
\begin{enumerate}
\item We say $\xc = (I, \mk)$ is a $\kappa$-context when it is a context and in addition if $\kappa$ is regular, $I$ is generated by 
$\kappa$ elements but not by $<\kappa$ elements.

\item We say $(\bar{I}, \xt)$ $\kappa$-represents $I \in \mk$ when $\bar{I} = \langle I_\zeta : \zeta < \kappa \rangle$ is increasing, 
$\xt = \{ \bar{t}_\zeta : \zeta < \kappa \rangle$, $\bar{t}_\zeta \subseteq I_\zeta$ is finite, $I_\zeta = \cl_I(\bar{t}_\zeta)$, 
$\bigcup_\zeta I_\zeta = I$.

\item $I \in \mk$ has a $\kappa$-representation iff $(I, \mk)$ is a $\kappa$-context for $\mk$ as usual. 

\end{enumerate}
\end{defn}

\begin{defn} \label{p5}
Let $\xc = (I, \mk)$ be a $\kappa$-context, $(\bar{I}, \xt)$ a $\kappa$-representation. We say a complete first-order theory $T$ is $\xc$-superstable when 
there is no $\preceq$-increasing $\langle M_\zeta : \zeta \leq \kappa \rangle$ sequence of models of $T$ and $p \in \ts(\bigcup_{\zeta < \kappa} M_\zeta)$ 
such that $p \rstr (M_{\zeta +1})$  $\xc$-shears over $M_\zeta$. 
\end{defn}

\begin{extn} \label{p8}
Let $\xc$ be a $\kappa$-context, $\xc = (I, \mk)$, $T$ first order complete, and $\lambda > \kappa$. 
$T$ is $\xc$-superstable ~iff~ for a dense set of $\Phi \in \Upsilon_\mk[T]$, $\GEM_{\tau(T)}(I, \Phi)$ is $\lambda$-saturated.  
\end{extn}

\begin{extn} \label{p14}
Let $T$ be complete, $\kappa$ regular. The following are equivalent:
\begin{enumerate}
\item $T$ is simple, $\kappa(T) \leq \kappa$.
\item For some $\kappa$-context $\xc$, $T$ is $\xc$-superstable. 
\end{enumerate}
\end{extn}

\noindent \textbf{Some questions}. Next we turn to some natural questions. 

\begin{qst}
Recalling $\ref{d:indr}$, shearing is strictly weaker than dividing, therefore {$($non-$)$}shearing cannot satisfy all the usual properties of 
independence relations in simple theories. 
Which such properties hold for shearing, and which fail? Does this depend on the context? 
\end{qst}

\begin{qst}
Does the analogue of ``forking = dividing'' hold for shearing in simple theories?
\end{qst}

For the next question, note that we may characterize the class of theories $T$ 
such that ``for any countable context $\xc$, $T$ is $\xc$-superstable if and only if 
the random graph is $\xc$-superstable'' indirectly by property $\ocirc$. Namely, this is the class of $T$ which are $\xc$-superstable for 
a countable context $\xc$ if and only if $\neg \ocirc_\xc$. 

\begin{qst}
Consider the class of theories $T$ such that ``for any countable context $\xc$, $T$ is $\xc$-superstable if and only if 
the random graph is $\xc$-superstable''.  Give an internal model-theoretic characterization of this class. 
\end{qst}

\begin{qst} \label{qst-order}
We may define an ordering on theories: $T_1 \leq T_2$ if for every countable context $\xc$, if $T_2$ is $\xc$-superstable then 
$T_1$ is $\xc$-superstable. What is the structure of this ordering?
\end{qst}

Question \ref{qst-order} requires both understanding theories and building contexts. 
The results of the present paper give some partial information:

\begin{lemma} In the ordering of $\ref{qst-order}$, 
\begin{itemize} 
\item The superstable theories are minimal, since every type in a superstable theory is definable (thus weakly definable) over a finite set, so it follows that such theories are $\xc$-superstable for every countable context $\xc$.  

\item The random graph is minimal among the unstable theories. $($If $T$ has the independence property, it will be susceptible to $\ocirc$, and 
if $T$ is not simple, see next item.$)$  

\item The non-supersimple theories are precisely those which are not $\xc$-superstable for any countable context $\xc$, see Theorem \ref{n5}. 
\end{itemize}
\end{lemma}

\begin{qst}
Is there a maximal class among the supersimple theories? 
\end{qst} 

One could eventually consider the analogue of \ref{qst-order} for contexts which are not necessarily countable, but for this it would make sense to 
first extend the results of this paper along the lines of $\S \ref{other-kappa}$ above.

\end{document}